\documentclass [11pt,oneside,a4paper,mathscr]{amsart}

\usepackage{amscd,amsmath,amssymb,euscript}
\usepackage[frame,cmtip,curve,arrow,matrix,line,graph]{xy}

\usepackage{todonotes}
\usepackage{url}

\date{\today}

\title{Derived automorphism groups of   K3 surfaces of Picard rank 1}
\author{Arend Bayer}
\address{School of Mathematics and Maxwell Institute,
The University of Edinburgh,
James Clerk Maxwell Building,
The King's Buildings, Mayfield Road, Edinburgh, Scotland EH9 3JZ,
United Kingdom}
\email{arend.bayer@ed.ac.uk}
\urladdr{http://www.maths.ed.ac.uk/~abayer/}
\author{Tom Bridgeland}
\address{Department of Pure Mathematics, The University of Sheffield, The Hicks Building, Hounsfield Road, Sheffield, England S3 7RH, United Kingdom.}
\email{t.bridgeland@shef.ac.uk}
\urladdr{http://www.tom-bridgeland.staff.shef.ac.uk/}
\dedicatory{We dedicate this paper to Professor Mukai on the occasion of his 60th birthday. }

\newtheorem{thm}{Theorem}[section]
\newtheorem{conj}[thm]{Conjecture}
\newtheorem{cor}[thm]{Corollary}
\newtheorem{prop}[thm]{Proposition}
\newtheorem{lemma}[thm]{Lemma}
\theoremstyle{definition}
\newtheorem{defn}[thm]{Definition}
\newtheorem{assumption}[thm]{Assumption}

\newtheorem{thm*}[thm]{Theorem$^*$}
\newtheorem{remark}[thm]{Remark}
\newtheorem{remarks}[thm]{Remarks}

\newtheorem*{example*}{Example}

\newcommand {\A}{\mathcal A}

\newcommand{\D}{{D}}
\newcommand {\C}{\mathbb C}
\newcommand{\CC}{\mathcal{C}}
\newcommand{\h}{\mathfrak{h}}
\newcommand{\uu}{\mathcal{V}(X)}

\newcommand {\HH}{{\mathcal H}}

\newcommand{\N}{\mathcal{N}}
\renewcommand{\O}{\mathcal{O}}
\renewcommand{\P}{\mathcal{P}}
\newcommand{\W}{\mathcal{W}}
\newcommand{\Q}{\mathcal{Q}}
\newcommand{\R}{\mathbb{R}}
\newcommand{\Z}{\mathbb Z}
\newcommand{\half}{\frac{1}{2}}
\newcommand{\M}{\mathcal{M}}
\renewcommand{\check}{{\vee}}
\renewcommand{\L}{{\mathcal{L}}}
\newcommand{\T}{\mathcal{T}}

\newcommand{\into}{\ensuremath{\hookrightarrow}}
\newcommand{\onto}{\twoheadrightarrow}
\newcommand{\lra}{\longrightarrow}
\newcommand{\isom}{\cong}
\newcommand{\tensor}{\otimes}
\newcommand{\lRa}[1]{\stackrel{#1}{\lra}}

\newcommand {\Aut}{\operatorname{Aut}}

\newcommand{\Coh}{\operatorname{Coh}}

\newcommand{\GL}{\operatorname{GL}}
\newcommand {\Hom}{\operatorname{Hom}}

\renewcommand{\Im}{\operatorname{Im}}
\newcommand{\Stab}{\operatorname{Stab}}
\newcommand{\NS}{\operatorname{NS}}
\renewcommand{\Re}{\operatorname{Re}}
\newcommand{\St}{\operatorname{Stab_{red}}}

\newcommand{\CY}{{\operatorname{CY}}}
\newcommand{\ST}{\operatorname{Tw}}
\newcommand{\PP}{{\mathbb{P}}}
\newcommand{\red}{{\operatorname{red}}}

\newcommand{\blank}{-}

\def\dbydt#1{\frac{d #1}{d t}}

\def\abs#1{\left\lvert#1\right\rvert}
\def\td{\mathop{\mathrm{td}}\nolimits}

\subjclass[2010]{14F05 (Primary); 14J28, 14J33, 18E30 (Secondary)}

\usepackage{pgf,tikz}
\usetikzlibrary{arrows}

\begin{document}
\begin{abstract}{ We give a complete description of the group of exact autoequivalences of the
bounded derived category of coherent sheaves on a K3 surface of Picard rank 1. We do this by
proving that  a distinguished connected component of the space of stability conditions  is preserved by all autoequivalences, and is contractible.}
\end{abstract}
\maketitle



\setcounter{tocdepth}{1}
\tableofcontents
\section{Introduction}

Let $X$ be a smooth complex projective variety. We denote by  $\D(X)=\D^b\Coh(X)$ the bounded derived category of coherent sheaves on $X$, and  by $\Aut \D(X)$ the group of triangulated, $\C$-linear autoequivalences of $\D(X)$, considered up to  isomorphism of functors. There is a subgroup 
\[ \Aut_{\operatorname{st}} \D(X) \isom \Aut X \ltimes \operatorname{Pic}(X) \times \Z\]
of $\Aut \D(X)$ whose elements are called  \emph{standard} autoequivalences: it is
the subgroup generated by the operations of  pulling back by automorphisms of $X$ and tensoring by line bundles, together with the shift functor.

The problem of computing the full group $\Aut \D(X)$ is usually rather difficult.  Bondal and Orlov
proved that when the canonical bundle $\omega_X$ or its inverse  is ample,  all autoequivalences are standard: $\Aut \D(X) =
\Aut_{\operatorname{st}} \D(X)$.
The group $\Aut \D(X)$ is also known explicitly  when $X$ is an abelian variety, due to work of
Orlov \cite{Orlov:abelian_equivalences}. 
Broomhead and Ploog \cite{Broomhead-Ploog:toric} treated many rational surfaces (including most
toric surfaces).  However, no other examples are known to date. 

The aim of this paper is to determine the group $\Aut \D(X)$ in the case when $X$ is a K3 surface of Picard rank 1.

\subsection*{Mukai lattice} For the rest of the paper $X$ will denote a complex algebraic K3 surface. In analogy to the strong
Torelli theorem, which describes the group $\Aut X$ via its action on $H^2(X)$, one naturally starts studying
$\Aut \D(X)$ via its action on cohomology. We will briefly review the relevant results,
see \cite[Section 10]{Huybrechts:FM} for more details.

The cohomology group
\[ H^*(X, \Z) = H^0(X,\Z)\oplus H^2(X,\Z)\oplus H^4(X,\Z),\]
comes equipped with a polarized weight two Hodge structure, whose algebraic part  is given by
\[
\N(X) = H^0(X, \Z) \oplus \NS(X) \oplus H^4(X, \Z), \quad \NS(X) = H^2(X,\Z)\cap H^{1,1}(X,\C),
\]
and whose polarization is given by the Mukai symmetric form
\[\langle (r_1,D_1,s_1),(r_2,D_2,s_2)\rangle =D_1\cdot D_2-r_1s_2 -r_2s_1.\]
The lattice $H^*(X,\Z)$ has signature $(4, 20)$, and  the subgroup $\N(X)$ has signature
$(2,\rho(X))$, where the Picard rank $\rho(X)$ is  the rank of the N\'eron-Severi  lattice $\NS(X)$. 

Any object of $\D(X)$ has a Mukai vector $v(E) = \mathrm{ch}(E) \sqrt{\td X}\in \N(X)$, and Riemann-Roch takes the form
\[\chi(E,F)=\big.\sum_{i\in \Z} (-1)^i \dim_\C \Hom^i_X(E,F)=-(v(E),v(F)).\]
Since any autoequivalence is of Fourier-Mukai type, the Mukai vector of its kernel induces a
correspondence; its action on cohomology preserves the Hodge filtration, the integral structure and the Mukai pairing. We
thus get a map
\[\varpi\colon \Aut \D(X)\lra \Aut H^*(X)\]
to the group of Hodge isometries.

 The group $\Aut H^*(X)$ contains an index 2 subgroup $\Aut^+ H^*(X)$ of Hodge isometries
preserving the orientation of positive definite 4-planes.
Classical results due to Mukai and Orlov (\cite{Mukai:BundlesK3, Orlov:representability}) imply that
the  image of $\varpi$ contains $\Aut^+ H^*(X)$, see \cite{HLOY:autoequivalences, Ploog:thesis}.
A much more difficult recent result due to Huybrechts, Macr{\`{\i}} and Stellari \cite{HMS:Orientation} is that
the image of $\varpi$  is contained in (and hence equal to) $\Aut^+ H^*(X)$. Our results in this
paper give an alternative, very different  proof of this fact in the case when $X$ has Picard
rank 1.

To determine the group $\Aut \D(X)$ it  thus remains to study the kernel of $\varpi$, which we will denote by $\Aut^0 \D(X)$.  This group is highly non-trivial due to the existence of spherical twist functors.
Recall that an object $S\in \D(X)$ is called spherical if \[\Hom_{\D(X)}(S,S[i])=\begin{cases} \C &\text{ if }i\in\{0,2\},\\ 0&\text{ otherwise.}\end{cases}\]
Associated to any such object  there is a corresponding  twist or reflection functor $\ST_S\in \Aut
\D(X)$, which appeared implicitly already in \cite{Mukai:BundlesK3}, and which was studied in detail
(and generalized) in \cite{Seidel-Thomas:braid}. 
The functor $\ST_S$ acts on cohomology by a reflection in the hyperplane orthogonal to $v(S)$, and hence its square  $\ST^2_S$ defines an element of the group $ \Aut^0 \D(X)$.

\subsection*{Stability conditions} Following the approach introduced by the second author in \cite{Bridgeland:K3}, we study
$\Aut^0 \D(X)$ using a second group action, namely its action on the space of stability conditions.

We denote by
$\Stab(X)$ the space of (full, locally-finite)  numerical stability conditions $(Z, \P)$ on
$\D(X)$. This is a finite-dimensional complex manifold with a faithful action of the group $\Aut \D(X)$. The central charge of a numerical stability condition takes the form
\[Z(\blank)=\big(\Omega,v(\blank)\big)\colon K(\D)\to \C \]
for some $\Omega \in \N(X) \otimes \C$,
and the induced forgetful map $\Stab(X) \to \N(X) \otimes \C$ is a local homeomorphism by \cite{Bridgeland:Stab}. 

Let $\Stab^\dag(X)\subset \Stab(X)$ be the connected component containing the set of geometric stability conditions,
for which all skyscraper sheaves  $\O_x$ are stable of the same phase.
The main result of \cite{Bridgeland:K3} is a description of this connected component,
which we now review.

Recall that $\N(X)$ has signature $(2, \rho(X))$.  Define the open subset
\[\P(X)\subset \N(X)\tensor \C\] consisting of vectors $\Omega\in
\N(X)\tensor \C$ whose real and imaginary parts span a positive definite 2-plane in $\N(X)\tensor
\R$.
This subset has two connected components, distinguished by the orientation induced on this 2-plane;
let $\P^+(X)$ to be the component containing vectors of the
form $(1,i\omega, -\half\omega^2)$ for an ample class $\omega\in \NS(X)\tensor \R$. Consider the
root system
 \[\Delta(X)=\{\delta\in \N(X)\colon (\delta,\delta)=-2\},\]
and  the corresponding hyperplane complement
\[\P^+_0(X)=\P^+(X)\setminus \bigcup_{\delta\in \Delta(X)} \delta^\perp.\]
We
note that $\Delta(X)$ is precisely the set of Mukai vectors of spherical objects in $\D(X)$.

\begin{thm}[{\cite[Theorem 1.1]{Bridgeland:K3}}] 
\label{bridge}The forgetful map sending a stability condition to the associated vector $\Omega\in \N(X)\tensor \C$ induces a covering map
\begin{equation} \label{eq:coverK3}
\pi\colon \Stab^\dagger(X)\lra \P^+_0(X),
\end{equation}
The covering is normal, and the group of  deck transformations can be identified with the subgroup of $\Aut^0 \D(X)$
which preserves the connected component $\Stab^\dagger(X)$.
\end{thm}

The Galois correspondence for the normal covering $\pi$ then gives a map
\[
\pi_1 \left(\P^+_0(X)\right) \to \Aut^0 \D(X),
\]
which is injective if and only if $\Stab^\dagger(X)$ is simply-connected, and
surjective if and only if $\Stab^\dagger(X)$ is preserved by $\Aut^0 \D(X)$.
This  suggests the  following conjecture of the second author:
\begin{conj}[{\cite[Conjecture 1.2]{Bridgeland:K3}}]
\label{kyoto}
The group $\Aut \D(X)$ preserves the connected component $\Stab^\dagger(X)$. Moreover, $\Stab^\dagger(X)$ is simply-connected. Hence  there is a short exact sequence of groups
\[1\lra \pi_1 (\P^+_0(X)) \lra \Aut \D(X) \lRa{\varpi} \Aut^+ H^*(X)\lra 1.\]
\end{conj}

Let us write $\Stab^*(X) \subset \Stab(X)$ to denote the union of those connected components which are
images of $\Stab^\dagger(X)$ under an autoequivalence of $\D(X)$. The content of Conjecture
\ref{kyoto} is then that the space $\Stab^*(X)$ should be connected and simply-connected.

\subsection*{Main result}
The main result of this paper is the following:

\begin{thm}
\label{main}
Assume that $X$ has Picard rank $\rho(X)=1$.  Then $\Stab^*(X)$ is contractible. In particular, Conjecture \ref{kyoto} holds in this case. 
\end{thm}

As has been observed previously by Kawatani \cite[Theorem 1.3]{Kawatani:hyperbolic},
when combined with a description of the fundamental group of $\P_0^+(X)$, Theorem \ref{main} implies:
 
\begin{thm}\label{corcor}Assume that $X$ has Picard rank $\rho(X)=1$. 
Then the group $\Aut^0 \D(X)$ is the product of $\Z$ (acting by even shifts) with the free group
generated by the autoequivalences $\ST^2_S$ for all  spherical vector bundles $S$. \end{thm}

As we will explain in Section \ref{sec:prelim},  the assumption $\rho(X)=1$ implies that any
spherical coherent sheaf $S$ on $X$ is necessarily a $\mu$-stable vector bundle.\smallskip

To prove Theorem \ref{main}, we start with the observation that the set of geometric stability
conditions is contractible (this easily follows from the results in \cite[Section
10--11]{Bridgeland:K3}). Now pick a point $x \in X$, and 
consider the width 
\[w_{\O_x}(\sigma) =\phi^+(\O_x)-\phi^-(\O_x),\] where $\phi^{\pm}(\O_x)$ is the
maximal and minimal phase appearing in the Harder-Narasimhan filtration of $\O_x$; this defines
a continuous function
\[w_{\O_x}\colon \Stab^*(X)\to \R_{\geq 0}.\]
We then construct a flow on $\Stab^*(X)$ which decreases $w_{\O_x}$, and use it to
contract $\Stab^*(X)$ onto the subset $w_{\O_x}^{-1}(0)$ of geometric stability conditions.

\begin{remarks}
\begin{itemize}
\item[(a)]
We do not currently know how to generalize our methods to higher Picard rank. Our entire argument  -- even the definition  of the flow -- is based on the fact that
the cone of classes in $\N(X) \otimes \R$ with negative square has two connected components.
Note that in the general case it is not known whether the universal cover of $\P^+_0(X)$ is
contractible, although this statement is implied by a special
case of a conjecture of Allcock, see \cite[Conjecture 7.1]{Allcock:completions}.

\item[(b)]
Many of the questions relevant to this article were first raised in \cite{Balasz:diffeomorphisms}.

\item[(c)]
There are various examples of derived categories of non-projective manifolds $Y$ for which 
it has been shown that a distinguished connected component of the space $\Stab(Y)$  is
simply-connected, see \cite{Ishii-Ueda-Uehara, Brav-HThomas:ADE, localP2, Qiu:Dynkin, Sutherland:CY3-algebra}.
In each of these cases, the authors used
the faithfulness of a specific group action on $\D(Y)$ to deduce simply-connectedness of (a
component of) $\Stab(Y)$, whereas our
logic runs in the opposite direction: a geometric proof of simply-connectedness implies the faithfulness of a group action.
\end{itemize}
\end{remarks}

\subsection*{Relation to mirror symmetry}
We will briefly explain the relation of Conjecture \ref{kyoto} to mirror symmetry; the details 
can be found in Section \ref{sect:MS}. The reader is also referred to  \cite{Bridgeland:spaces} for more details on this. The basic point is that the group of autoequivalences of $\D(X)$ as a Calabi-Yau category coincides with the fundamental group of a mirror family of K3 surfaces.



A stability condition $\sigma\in \Stab^*(X)$ will be called \emph{reduced} if the corresponding vector $\Omega\in \N(X)\tensor \C$ satisfies $(\Omega,\Omega)=0$. This condition defines a complex submanifold
\[\Stab^*_{\operatorname{red}}(X)\subset \Stab^*(X).\]
As explained in  \cite{Bridgeland:spaces}, this is the first example of Hodge-theoretic conditions on stability conditions: it is not known how such a submanifold should be defined for higher-dimensional Calabi-Yau categories.

Define  a subgroup \[\Aut^+_\CY H^*(X)\subset \Aut^+ H^*(X)\]
consisting of Hodge isometries whose complexification acts trivially on the complex  line $H^{2,0}(X,\C)$, and let 
\[\Aut_{\rm CY} \D(X) \subset \Aut \D(X)\]
denote the subgroup of autoequivalences $\Phi$  for which $\varpi(\Phi)$ lies in $\Aut^+_\CY H^*(X)$. Such autoequivalences are usually called \emph{symplectic}, but we prefer the term \emph{Calabi-Yau} since, as we explain in the Appendix, this condition is equivalent to the statement that $\Phi$ preserves all 
 Serre duality pairings
\[\Hom^i_{X}(E,F)\times \Hom^{2-i}_X(F,E)\lra \C.\]

Let us now consider  the orbifold quotient
\[\L_{\rm Kah}(X)=\Stab_{\rm red}^*(X)/\Aut_{\rm CY} D(X).\]
There is a free action of the group $\C$ on $\Stab^*(X)$, given by rotating the central charge $Z$ and adjusting the phases of stable objects in the obvious way.  The action of $2n\in \Z\subset \C$ coincides with the action of the shift functor $[2n]\in \Aut_{\rm CY} \D(X)$. In this way  we obtain 
  an action of $\C^*=\C/2\Z$ on the space $\L_{\rm Kah}(X)$, and we can also consider the quotient
\[ \M_{\rm Kah}(X)=\L_{\rm Kah}(X)/\C^*.\]
 We view this complex orbifold as  a mathematical version of the stringy K{\"a}hler moduli space of the K3 surface $X$.


 Using Theorem \ref{bridge} one easily deduces the following more concrete description for this orbifold.
Define  period domains
\[\Omega^\pm(X) = \left\{\Omega \in \PP( \N(X) \otimes \C) \colon (\Omega, \Omega) = 0, (\Omega, \overline{\Omega}) > 0\right\},\]
\[\Omega^\pm_0(X)=\Omega^\pm(X)\setminus \bigcup_{\delta\in \Delta(X)} \delta^\perp,\]
and let 
 $\Omega^+_0(X)\subset \Omega_0(X)$ be the connected component containing classes  $(1,i\omega,\half \omega^2)$ for $\omega\in \NS(X)$ ample. Then there is an identification
\[\M_{\rm{Kah}}(X) =\Omega_0^+(X)/\Aut_{\rm CY} H^*(X).\]

The mirror phenomenon in this context is the fact that this orbifold also arises as the base of a family  of lattice-polarized K3 surfaces. More precisely, under mild assumptions (which always hold when $X$ has Picard number $\rho(X)=1$), the stack $\M_{\rm Kah}(X)$ can be identified with
 the base of Dolgachev's family of lattice-polarized K3 surfaces mirror to
$X$ \cite{Dolgachev:MSK3}.

Conjecture \ref{kyoto} is equivalent to the statement that the natural map
\[\pi_1^{\rm orb} \big(\M_{\rm Kah}(X))\lra \Aut_{\rm CY} \D(X)/[2]\]
is an isomorphism.
Our verification of this Conjecture in the case $\rho(X)=1$   thus gives a precise
incarnation of Kontsevich's general principle that the group of derived autoequivalences of a Calabi-Yau variety should be related to the fundamental group of the base of the mirror family.

\subsection*{Acknowledgements.} Thanks to Chris Brav, Heinrich Hartmann, Daniel Huybrechts, Ludmil
Katzarkov, Kotaro Kawatani and Emanuele Macr{\`{i}} for useful discussions. Special thanks are due to Daniel
Huybrechts for persistently encouraging us to work on this problem. We are also grateful to 
the referees for a very careful reading of the article, and to Ivan Smith for pointing out an inaccuracy in an earlier version. 
The first author was supported by ERC grant WallXBirGeom 337039 while completing work on this article.

It is a pleasure to dedicate this article to Professor Mukai on the occasion of his 60th birthday.
The intellectual debt this article owes to his work, starting with \cite{Mukai:BundlesK3}, can't be
overstated.

\section{Preliminaries}
\label{sec:prelim}

Let $X$ be a complex algebraic K3 surface.

\subsection*{Stability conditions} \label{subsec:stab}
Recall that a \emph{numerical stability condition} $\sigma$ on $X$ is a pair $(Z, \P)$ where
$Z \colon \N(X) \to \C$ is a group homomorphism, and \[\P=\bigcup _{\phi\in \R} \P(\phi) \subset \D(X)\]
is a full subcategory. The homomorphism $Z$ is called the \emph{central charge}, and the objects of the subcategory $\P(\phi)$  are said to be \emph{semistable of phase $\phi$}. We refer to \cite{Bridgeland:Stab} and \cite[Section 2]{Bridgeland:K3} for a
complete definition. Any object $E \in \D(X)$ admits a unique \emph{Harder-Narasimhan (HN) filtration}
\[
\xymatrix@C=.5em{
0_{\ } \ar@{=}[r] & E_0 \ar[rrrr] &&&& E_1 \ar[rrrr] \ar[dll] &&&& E_2
\ar[rr] \ar[dll] && \ldots \ar[rr] && E_{n-1}
\ar[rrrr] &&&& E_n \ar[dll] \ar@{=}[r] &  E_{\ } \\
&&& A_1 \ar@{-->}[ull] &&&& A_2 \ar@{-->}[ull] &&&&&&&& A_n \ar@{-->}[ull] 
}
\]
with $A_i\in \P(\phi_i)$ semistable, and $\phi_1 > \phi_2 > \dots > \phi_n$. We  refer to the objects $A_i$ as the \emph{semistable factors} of $E$. We
write $\phi_+(E) = \phi_1$ and $\phi_-(E) = \phi_n$ for the maximal and minimal phase appearing in the
HN filtration  respectively.

Using the non-degenerate Mukai pairing on $\N(X)$ we can write the central charge of  any numerical stability condition in the form
\[Z(\blank)=(\Omega,v(\blank))\colon \N(X)\to \C \]
for some uniquely defined $\Omega \in \N(X) \otimes \C$. 
Fix a norm $\|\cdot\|$ on the finite-dimensional vector space $\N(X)\tensor \R$. A numerical
stability condition $\sigma=(Z,\P)$ is said to satisfy the \emph{support condition} \cite{Kontsevich-Soibelman:stability}
if there is a constant $K>0$ such that for any semistable object $E\in\P(\phi)$ there is an inequality
\[|Z(E)|\geq K\cdot \|E\|.\]
 As shown in \cite[Proposition B.4]{localP2}, this is equivalent to the condition that $\sigma$ be  \emph{locally-finite} \cite[Defn. 5.7]{Bridgeland:Stab} and \emph{full} \cite[Defn. 4.2]{Bridgeland:K3}.

If the stability condition $\sigma=(Z,\P)$  is locally-finite, each subcategory $\P(\phi)$ is a finite length abelian category; the simple objects of $\P(\phi)$ are said to be \emph{stable of phase $\phi$}. Each semistable factor $A_i$ of a given object $E\in \D(X)$ has a Jordan-H{\"o}lder filtration in $\P(\phi_i)$. Putting these together gives a (non-unique) filtration of $E$ whose factors $S_j$ are stable, with phases taken from the set $\{\phi_1,\cdots,\phi_n\}$. These  objects  $S_j$ are uniquely determined by $E$ (up to reordering and isomorphism); we refer to them as the \emph{stable factors} of $E$.

We let $\Stab(X)$ denote the set of all numerical stability conditions on $\D(X)$ satisfying the support condition. This set has a natural topology induced by a (generalized) metric $d(\blank,
\blank)$. We refer to \cite[Proposition 8.1]{Bridgeland:Stab} for the full definition, and simply list  the following properties:
\begin{itemize}
\item[(a)] For any object $E\in \D(X)$, the functions $\phi^\pm(E)\colon \Stab(X)\to \R$ are continuous.\smallskip
\item[(b)]Take $0<\epsilon<1$ and consider  two stability conditions  $\sigma_i = (\P_i, Z_i)$  such that
$d(\sigma_1, \sigma_2) < \epsilon$. Then if an object $E\in \D(X)$ is semistable in one of the stability conditions $\sigma_i$, the arguments of the
complex numbers $Z_i(E)$ differ by at most $\pi \epsilon$.\smallskip

\item[(c)]  The forgetful map
$\Stab(X)\to \N(X)\tensor \C$ sending a stability condition to the vector $\Omega$
is a local homeomorphism \cite{Bridgeland:Stab}.
\end{itemize}

Let $\GL^+_2(\R)$ be the group of orientation-preserving automorphisms of $\R^2$. The universal cover $\operatorname{\widetilde{GL^+_2}}(\R)$ of this group acts on $\Stab(X)$ by post-composition on the central charge $Z\colon  \N(X)\to \C\isom \R^2$ and a suitable relabelling of the phases (see  \cite{Bridgeland:Stab}). There is a subgroup $\C \subset \widetilde{\GL^+_2}(\R)$ which acts freely; explicitly this action is given by $\lambda\cdot(Z, \P) = (Z', \P')$ with
$Z' = e^{\pi i \lambda} \cdot Z$ and $\P'(\phi) = \P(\phi - \Re \lambda)$.
 There is also  an action of the group $\Aut \D(X)$ on $\Stab(X)$ by $\Phi\cdot (Z,\P)=(Z',\P')$ with $Z'=Z\circ \varpi(\Phi)^{-1}$ and $\P'(\phi)=\Phi(\P(\phi))$.



\subsection*{Period domains}
Recall the definitions of the open subsets
\[\P^\pm_0(X)\subset \P^\pm(X)\subset \N(X)\tensor \C\]
from the introduction. Now consider the corresponding subsets
\[\Q^\pm(X) = \left\{\Omega \in \N(X) \otimes \C\colon (\Omega, \Omega) = 0, (\Omega, \overline{\Omega}) > 0\right\},\]
\[\Q^\pm_0(X)=\Q^\pm(X)\setminus \bigcup_{\delta\in \Delta(X)} \delta^\perp.\]
These are invariant under the rescaling action of $\C^*$ on $\N(X)\tensor \C$. As with $\P^\pm(X)$, the subset  $\Q^\pm(X)$ consists of two connected components, and we let $\Q^+(X)=\Q^\pm(X)\cap \P^+(X)$ be the one containing  classes  $(1,i\omega,\half \omega^2)$ for $\omega\in \NS(X)$ ample.
 
The normalization condition
$(\Omega, \Omega) = 0$ is equivalent to the statement that
$\Re \Omega, \Im \Omega$ are a conformal basis of the 2-plane in $\N(X)\tensor \R$ which they span:
\[ (\Re \Omega, \Im \Omega) = 0, \quad 
(\Re \Omega)^2 = (\Im \Omega)^2 > 0. \]
From this, one easily sees that each $\GL_2(\R)$-orbit in $\P^\pm(X)$ intersects
$\Q^\pm(X)$  in a unique $\C^*$-orbit.  It follows that
\begin{equation} \label{eq:domainquotients}
\P^+_0(X)/\GL^+_2(\R) = \Q^+_0(X)/\C^*,
\end{equation}
and further that $\Q^+_0(X) \subset \P^+_0(X)$ is a deformation retract.

 Let $\T^\pm(X) \subset \N(X) \otimes \C$ be the two components of the tube domain
\[
\T^\pm(X) = \left\{\beta + i \omega \colon \beta, \omega \in \NS(X)\tensor \R, (\omega, \omega) > 0 \right\},
\]
where $\T^+(X)$ denotes the component containing $i \omega$ for ample classes $\omega$.  
The map $\Omega = \exp(\beta + i \omega)$ defines an embedding
\begin{equation}
\label{sect}
\T^+(X) \hookrightarrow \Q^+(X)\end{equation}
which gives a section of the $\C^*$-action on $\Q^+(X)$; this identifies $\T^+(X)$ with the
quotients in \eqref{eq:domainquotients}. We set
\[\T^+_0(X)=\T^+(X)\cap \Q^+_0(X)\]
for the corresponding hyperplane complement.

Consider the case when $X$ has Picard number $\rho(X) = 1$. The ample generator of $\NS(X)$ allows us to identify
$\T^+(X)$ canonically with the upper half plane $\h$.  
The hyperplane complement $\T_0^+(X) \subset \T^+(X)$ then corresponds to the open subset
\[\h^o=\h\setminus \{\beta+i\omega\in \h\colon \langle \exp({\beta+i\omega}),\delta\rangle=0 \text{ when }
\delta\in \Delta(X)\}.\]
We now recall briefly the description of the fundamental group $\pi_1(\P_0^+(X))$ given in
\cite[Prop. 2.14]{Kawatani:hyperbolic}. 
The system of hyperplanes $\bigcup_{\delta\in \Delta(X)}\delta^\perp\subset \Q^\pm(X)$ is locally-finite, and it follows that the complement  $\h \setminus \h^o$ is a discrete subset of  the upper half plane. From this one deduces
that the fundamental group
$\pi_1(\h^0)$ is the free group with the obvious generators.\footnote{Let us briefly sketch a
proof. For $\epsilon > 0$, let $\h^o_\epsilon = \h^0 \cap \{\Im z > \epsilon, |\Re z| <1/\epsilon\}$.
Then $h^o_\epsilon$ is homeomorphic to a disc with finitely many holes; by Seifert-van Kampen, its
fundamental group is the free group with finitely many generators. On the other hand, 
using compactness of loops and homotopies one can show that $\pi_1(\h^o)$ is the union of the
fundamental groups $\pi_1(\h^o_\epsilon)$ as $\epsilon \to 0$.}
 Now $\P_0^+(X)$ is a $\GL^+_2(\R)$-bundle over $\h^0$, and this bundle is trivial since the map
\eqref{sect} defines a section. This yields
\begin{equation} \label{eq:pi1P0}
\pi_1(\P_0^+(X)) = \Z \times \pi_1(\h^o).
\end{equation}

\subsection*{Geometric stability conditions}

A stability condition in $\Stab(X)$ is  said to be \emph{reduced} if the corresponding vector  $\Omega\in \N(X)\tensor \C$ satisfies $(\Omega,\Omega)=0$. The set of reduced stability conditions forms a complex submanifold
\[\Stab_{\rm red}(X)\subset \Stab(X).\]
This submanifold preserves all topological information:

\begin{lemma} \label{reducedretract}
The inclusion 
$\Stab_{\rm red}(X)\subset \Stab(X)$ is a deformation retract.
\end{lemma}
\begin{proof}
The action of $\widetilde{\GL_2^+(\R)}$ on $\Stab(X)$ is free, and  $\Stab_{\rm red}(X)$ is invariant under the subgroup $\C \subset
\widetilde{\GL_2^+(\R)}$, and contains exactly one $\C$-orbit for every
$\widetilde{\GL_2^+(\R)}$-orbit of $\Stab^*(X)$. The result then follows from the contractibility of the quotient space $\widetilde{\GL_2^+(\R)}/\C\isom\GL_2^+(\R)/\C^*\isom  \h$.
\end{proof}

Recall from the introduction that we denote by $\Stab^\dag(X)\subset \Stab(X)$ the connected component constructed in \cite{Bridgeland:K3}, and by $\Stab^*(X)\subset \Stab(X)$ the union of those connected components which are images of $\Stab^\dag(X)$ under elements of $\Aut \D(X)$.
Restricted to the components $\Stab^*(X)\subset \Stab(X)$, the reduced condition is precisely that the vector $\Omega$ lies in $\Q^+(X)\subset \P^+(X)$.

The starting point in the
description of $\Stab^\dag(X)$ given in \cite{Bridgeland:K3} is a characterization of the set of
stability conditions $U(X) \subset \Stab_{\rm red}(X)$, for which all skyscraper sheaves $\O_x$ of points $x \in X$ are stable of the
same phase. Such stability conditions are called \emph{geometric}.
Note that the subset $U(X)$ is invariant under the $\C$-action on $\Stab(X)$, and each orbit contains a unique stability condition for which the  objects $\O_x$ are stable of phase 1.

To describe the set $U(X)$ we first define
\[\Delta^+(X)=\{(r,D,s)\in \Delta: r>0\}\subset \Delta(X)\]
and consider the open subset
\[\uu=\left\{\beta+i\omega\in T^+(X): \delta \in \Delta^+(X) \implies \langle \exp({\beta+i\omega}), \delta\rangle \notin
\R_{\le 0} \right\}.\]
The following result is proved in \cite[Sections 10-12]{Bridgeland:K3}.

\begin{thm} \label{geom}
The forgetful map $\Stab(X) \to \N(X)\tensor \C$ induces a bijection between the set of reduced, geometric stability conditions in which the objects $\O_x$ have phase 1, and the set of vectors of the form
$\Omega = \exp(\beta + i\omega)$ with $\beta + i \omega \in \uu$. Thus there is an isomorphism
\[
U(X) \cong \C \times \uu.
\]
\end{thm}

Let us again consider the case when $X$ has Picard number $\rho(X)=1$.
Note that we then have
\[\Delta(X)=\Delta^+(X) \sqcup -\Delta^+(X),\]
since there are no spherical classes of the form $(0,D,s)$ because the intersection form on $\NS(X)$ is positive definite.
 The subset $\uu\subset \h^0$ is obtained by removing the vertical line segment
between the real line and each hole $\h \setminus \h^o$, see Figure \ref{fig:hwithoutsegments}.
For each $\delta \in \Delta^+(X)$  there is a unique spherical sheaf
$S_\delta \in
\Coh X$ with $v(S_\delta) = \delta$, and this sheaf $S_\delta$ is automatically a $\mu$-stable vector
bundle.\footnote{The existence is part of \cite[Theorem 0.1]{Yoshioka:Irreducibility}. Mukai already
proved that a spherical torsion-free sheaf is automatically locally free and $\mu$-stable see
\cite[Prop.~3.3 and Prop.~3.14]{Mukai:BundlesK3}; the torsion-freeness in the case $\rho(X) = 1$
follows with the same argument, see Remark \ref{rem:sphericaltorsionfree}. Finally, the uniqueness
is elementary from stability, see \cite[Corollary 3.5]{Mukai:BundlesK3}.}

\begin{figure}
\includegraphics[scale=0.7]{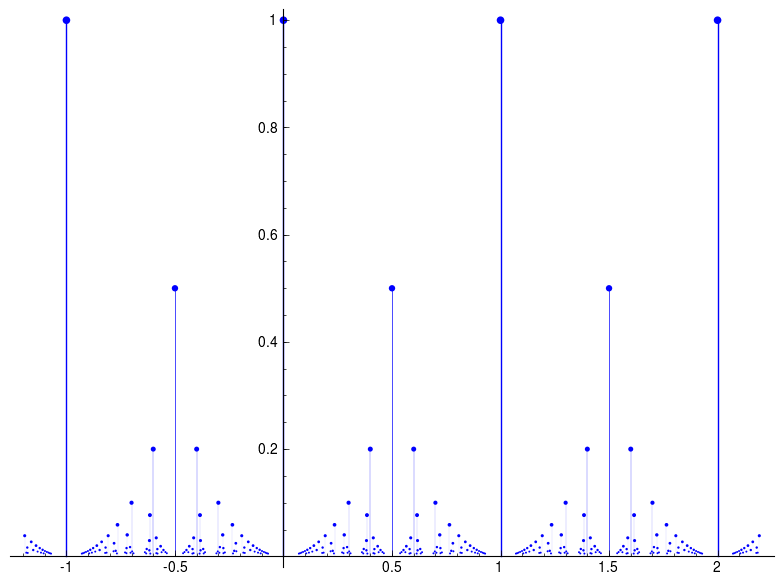}
\caption{\label{fig:hwithoutsegments}The subsets $\uu \subset \h^o \subset \h$ of the upper half plane, for a generic K3 surface
that is a double-cover of $\mathbb P^2$.}
\end{figure}

The following result was proved by Kawatani \cite[Prop. 5.4]{Kawatani:hyperbolic}.  For the reader's convenience we include a sketch proof here.

\begin{prop} \label{prop:loopandtwist}
Let $\delta \in \Delta^+(X)$ be a spherical class with positive rank, and $S_\delta$ the spherical
vector bundle with Mukai vector $\delta$.  The deck transformation of the normal covering $\pi$
associated to the anti-clockwise loop in $\h^o$ around $\delta^\perp$ 
is the square $\ST_{S_\delta}^2$ of the twist functor associated to $S_\delta$.
\end{prop}
Here we have implicitly chosen a geometric stability condition as base point
of $\Stab^\dagger(X)$, so that we can consider the covering map $\pi$ as a map of
pointed topological spaces.

\begin{proof}
Given the line segment in $\h \setminus \uu$ associated to such $\delta$, there are
two corresponding walls $\W^+_\delta$ and $W^-_\delta$ of the geometric chamber,
depending on whether we approach the line segment
from the left or the right, respectively. Let $S_\delta$ be the corresponding spherical vector
bundle, with $Z(S_\delta) \in \R_{<0}$ for $(Z, \P) \in W^\pm_\delta$. These walls are described
by the cases $(A^+)$, $(A^-)$ of \cite[Theorem 12.1]{Bridgeland:K3}, respectively (where  the
vector bundle $A$ in the citation is exactly our vector bundle $S_\delta$).

In the proof of \cite[Proposition 13.2]{Bridgeland:K3}, it is shown that crossing these walls
leads into the image of the geometric chamber under the spherical twist
$\ST_{S_\delta}^{\pm 2}$. 
Now consider the loop $\gamma$ around the hole corresponding to $\delta$, and lift it to a path starting
in the geometric chamber $U(X)$. It follows from the preceding discussion that its endpoint will
lie in $\ST_{S_\delta}^{\pm 2} U(X)$, with the sign depending on the orientation of the loop. 
\end{proof}


\subsection*{Rigid and semirigid objects}
The following important definition  generalizes the notions of rigid and semirigid coherent sheaves from \cite{Mukai:BundlesK3}:

\begin{defn}
An object  $E\in \D(X)$ will be  called rigid or semirigid if 
\begin{itemize}
\item[(a)] $\dim_{\C} \Hom^1_X(E,E) =0 \text{ or }2$,
respectively, and

\item[(b)] $\Hom^i_X(E,E)=0 \text{ for all }i<0$.
\end{itemize}
\end{defn}

We say that an object $E\in \D(X)$ is (semi)rigid if it is either rigid or semirigid.
It follows from Riemann-Roch and Serre duality that if $E$ is (semi)rigid, then the Mukai vector $v(E)$ satisfies $(v(E),v(E))\leq 0$, with strict inequality in the rigid case. 

We will need a derived category version of \cite[Corollary 2.8]{Mukai:BundlesK3}:

\begin{lemma}
\label{stan}
Suppose $A\to E\to B$ is an exact triangle and $\Hom_X(A,B)=0$. Then
\[\dim_\C \Hom^1_X(E,E) \geq \dim_\C \Hom^1_X(A,A) + \dim_\C \Hom^1_X(B,B).\]
\end{lemma}

\begin{proof}
Consider the space $V$ of maps of triangles
\begin{equation*}\begin{CD}
 A& @>f>> &E &@>g>> &B&@>h>> &A[1] \\
@V{\alpha}VV && @V{\gamma}VV && @V{\beta}VV && @V{\alpha[1]}VV\\
A[1] &@>f[1]>> &E[1]&@>g[1]>> &B[1] &@>h[1]>>
&A[2].\end{CD}\end{equation*} 
There are obvious maps
\[F\colon V\to \Hom^1_X(A,A)\oplus \Hom^1_X(B,B), \quad G\colon V\to \Hom^1_X(E,E).\]
The result follows from the two claims that $F$ is surjective and $G$ is injective.

For the first claim, note that given maps $\alpha\colon A\to A[1]$ and $\beta\colon B\to B[1]$ we obtain a map of triangles as above, because the difference
\[h[1]\circ \beta - \alpha[1]\circ h\colon B\to A[2]\]
 vanishes by the assumption and Serre duality. For the second claim, a simple diagram chase using the assumption $\Hom_X(A,B)=0$ shows that  any map of triangles as above in which $\gamma=0$ is necessarily zero.
\end{proof}

The next result is a consequence of \cite[Proposition 2.9]{HMS:generic_K3s}. For the reader's convenience we include the easy proof here.

\begin{lemma}
\label{stab}
Let $\sigma$ be a stability condition on $\D(X)$, and $E$ an object of $\D(X)$.
\begin{itemize}
\item[(a)]
If $E$ is  rigid, then all stable factors of $E$ are rigid.
\item[(b)] If $E$ is semirigid, then all stable factors of $E$ are rigid or semirigid, and at most one of them is semirigid.
\end{itemize}
\end{lemma}

\begin{proof}
First note that applying Lemma \ref{stan} repeatedly to the HN filtration of $E$ allows us to reduce to the case when  the object $E\in \P(\phi)$ is in fact semistable.

If $E$ has more than one non-isomorphic stable factor, then by taking a maximal subobject whose Jordan-H{\"o}lder factors are all isomorphic, we can find a nontrivial short exact sequence \[0\lra A\lra E\lra B\lra 0\] in $\P(\phi)$ with $\Hom_X(A,B)=0$. Applying  Lemma \ref{stan} shows that $A$ and $B$ are also either rigid or semirigid, with at most one being semirigid, and we can then proceed by induction on the length of $E$ in the finite length category $\P(\phi)$.

Thus we can assume that all stable factors of $E$ are the same object $S$. If $E$ is stable the claim is trivial. Otherwise there are non-identity maps $E\to E$ obtained by factoring through a copy of $S$. In particular $\dim_\C \Hom_X(E,E)>1$. Since $E$ is semirigid it follows that $v(E)^2<0$. But $v(E)^2=n^2 \cdot v(S)^2$ so also  $v(S)^2<0$, and it follows from Riemann-Roch and Serre duality  that $S$ is rigid. \end{proof}


\subsection*{Wall-crossing}
\label{wallcr}
Fix a stability condition $\sigma_0=(Z_0,\P_0)\in \Stab(X)$ and a phase $\phi\in \R$. Recall that the category $\P_0(\phi)$ of semistable objects of phase $\phi$ is a finite length abelian category whose simple objects are the stable objects of phase $\phi$. Let  us now fix a Serre subcategory $\A\subset \P_0(\phi)$: this corresponds to choosing some subset of the stable objects of phase $\phi$ and considering only those objects of $\P_0(\phi)$  whose stable factors lie in this subset. Let \[U=B_{\frac{1}{2}}(\sigma_0)\subset \Stab(X)\] be the ball of radius $\frac{1}{2}$ at $\sigma_0$ with respect to the standard metric on $\Stab(X)$.  
Given a stability condition $\sigma=(Z,\P)\in U$, the restriction of the central charge $Z$ defines a (tilted) stability function on the abelian category $\A$. We then have two notions of stability for objects $F\in \A$, namely stability  with respect to the stability condition  $\sigma$, and stability with respect to the stability function $Z$ on the abelian category $\A$.
The following useful result addresses the relationship between these two notions.

\begin{lemma}
\label{reduce}
For every $R>0$ there is a neighbourhood $\sigma_0\in U(R)\subset U$ with the following property: if $\sigma=(Z,\P)\in U(R)$ and $F\in \A$ satisfies $|Z_0(F)|<R$, then $F$ is $\sigma$-semistable (resp. $\sigma$-stable) precisely if it  is $Z$-semistable (resp. $Z$-stable).
\end{lemma}

\begin{proof}
First consider an arbitrary  $\sigma=(Z,\P)\in U$. If $F\in \A$ is $Z$-unstable, then there is a short exact sequence
\[0\lra A \lra F \lra B\lra 0\]
in $\A$ such that $\phi(A)>\phi(F)>\phi(B)$. Since $\sigma\in U$ there is an inclusion $\A\subset
\P(\phi-\frac{1}{2},\phi+\frac{1}{2})$, and  it follows easily that $F$ is also $\sigma$-unstable
(see \cite[Prop.~5.3]{Bridgeland:Stab}).

For the converse, take $0<\epsilon<\frac{1}{8}$ and consider triangles 
\begin{equation}
\label{trilemma}A \lra F \lra B\lra A[1]\end{equation}
in $\D(X)$, all of whose objects lie in the subcategory $\P_0(\phi-2\epsilon,\phi+2\epsilon)$, and such that $F\in \A$ satisfies $|Z_0(F)|<R$. Semistability of $F$ in $\sigma_0$  ensures that for any such triangle  there are inequalities
\begin{equation}
\label{inequ}\phi_0(A)\leq\phi_0(F)\leq\phi_0(B),\end{equation}
where $\phi_0$ denotes the phase function for the stability condition $\sigma_0$.
 Moreover, it is easy to see that if equality holds in \eqref{inequ} then  $A,B\in \P_0(\phi)$, and hence, in fact,   $A,B\in \A$.

By the support property for $\sigma_0$, the set of possible Mukai vectors $v(A)$ and $v(B)$ is finite. Thus we can choose an open neighbourhood $\sigma_0\in U(R)\subset U$ small enough so that whenever the inequality \eqref{inequ} is strict, the same inequality of phases holds  for all $\sigma\in U(R)$. We can also assume that $U(R)\subset B_\epsilon(\sigma_0)$ is contained in the $\epsilon$-ball centered at $\sigma_0$.

Now suppose that $F\in \A$ satisfies $|Z_0(F)|<R$ and take a stability condition $\sigma=(Z,\P)\in U(R)$.
Then $F\in \P(\phi-\epsilon,\phi+\epsilon)$. Suppose that $F$ is  $\sigma$-unstable. Then we
can find a  triangle \eqref{trilemma} with all objects lying in $\P(\phi-\epsilon,\phi+\epsilon)$,
and such that $\phi(A)>\phi(F)>\phi(B)$.  All objects of this triangle then  lie in
$\P_0(\phi-2\epsilon,\phi+2\epsilon)$, so our assumption ensures that equality holds in \eqref{inequ} and hence  $A,B\in \A$. It follows  that   $F$ is  also $Z$-unstable.

We have now proved that $F$ is $Z$-semistable precisely if it is $\sigma$-semistable. A similar argument which we leave to the reader now shows furthermore that $F$ is $Z$-stable precisely if it is $\sigma$-stable. 
 \end{proof}

\begin{remark}
It follows immediately that for objects $F\in \A$ with $|Z_0(F)|<R$ and for stability conditions $\sigma=(Z,\P)\in U(R)$, the HN filtration of $F$ with respect to $\sigma$ coincides with the HN filtration of $F$ in $\A$ with respect to the stability function $Z$ (note that this latter filtration automatically exists because $\A\subset \P(\phi)$ is of finite length). A similar remark applies to Jordan-H{\"o}lder filtrations. 
\end{remark}

 In studying wall-crossing behaviour, the following definition is often useful.

\begin{defn}
An object $F\in \D(X)$  is said to be quasistable in a stability condition $\sigma$ if it is semistable, and all its stable factors have Mukai vectors lying on the same ray $\R_{>0} \cdot v\subset \N(X)\tensor \R$.
\end{defn}

Note that if $v(F)\in \N(X)$ is primitive, then $F$ is quasistable precisely if it is stable.  The
following result is a mild generalization of  \cite[Prop.~9.4]{Bridgeland:K3}, and can be proved using the same argument given there. Instead we give an easy proof using Lemma \ref{reduce}.

\begin{prop}
\label{open}
The set of points $\sigma\in \Stab(X)$ for which a given object $F\in \D(X)$ is stable (respectively quasistable) is open.
\end{prop}

\begin{proof}
Let $F$ be semistable in some stability condition $\sigma_0=(Z_0,\P_0)$. Choose $R>|Z_0(F)|$  and
apply Lemma \ref{reduce} with $\A\subset \P(\phi)$ being the abelian subcategory generated by the stable factors of $F$. When $F$ is quasistable  all these stable factors have proportional Mukai vectors, so the  stability functions on $\A$ induced by stability conditions in $U(R)$  map $K(\A)$ onto a line in $\C$. For such stability functions all objects of $\A$ are $Z$-semistable, and an object is $Z$-stable precisely if it is simple. The result therefore follows from Lemma \ref{reduce}.
\end{proof}


\section{Walls and chambers}
\label{sec:walls}

From now on, let $X$ be a complex projective K3 surface of Picard rank 1.  We shall also fix a (semi)rigid object $E\in \D(X)$ satisfying
\[\Hom_X(E,E)=\C.\] Eventually, in Section \ref{sec:conclusion}, we will take $E$ to be  a skyscraper sheaf $\O_x$.

\subsection*{Stable objects of the same phase}

Let $\sigma\in \Stab(X)$ be a stability condition.
We first gather some simple results about the relationship between (semi)rigid stable objects in $\sigma$  of some fixed phase.

\begin{lemma}
\label{obv}
Suppose that $S_1,S_2$ are non-isomorphic stable objects of the same phase, at least one of which is rigid. Then the  Mukai vectors $v(S_i)$ are linearly independent in $\N(X)$.  \end{lemma}

\begin{proof}
Since the $S_i$ are stable of the same phase, $\Hom_X(S_i,S_j)=0$ for $i\neq j$, and Serre duality and Riemann-Roch then show that $(v_1,v_2)\geq 0$. Suppose there is a non-trivial linear relation between the  vectors $v_i=v(S_i)$.  Since the central charges $Z(S_i)$ lie on the same ray it must take the form $\lambda_1 v_1 = \lambda_2 v_2$ with $\lambda_1,\lambda_2>0$.   But then $( v_i,v_i)\geq 0$ which contradicts the assumption that one of the $S_i$ is rigid.
\end{proof}

The Mukai vectors of rigid and semirigid objects are contained in the cone
\[\CC=\{v\in \N(X)\tensor \R\setminus\{0\} \colon (v,v)\leq 0\}.\]
Since  $X$ has Picard number $\rho(X)=1$,  the lattice $\N(X)$ has signature $(2,1)$, and $\CC$
 is therefore a disjoint union of two connected  components $\CC^\pm$ exchanged by the inverse map $v\mapsto -v$. By convention we take $\CC^+$ to be the component containing the class $(0,0,1)$.
The following elementary observation will be used  frequently:

\begin{lemma}
\label{z}
Suppose $\alpha,\beta\in \CC^+$. Then $(\alpha,\beta)\leq 0$. Moreover
\[(\alpha,\beta)=0 \implies  (\alpha,\alpha)=0=(\beta,\beta),\]
in which case $\alpha, \beta$ are proportional.
\end{lemma}

\begin{proof}
We can take co-ordinates $(x,y,z)$ on $\N(X)\tensor \R\isom \R^3$ so that the quadratic form associated to $(-,-)$  is $x^2+y^2-z^2$. Then $\CC$ is the set of nonzero vectors with $x^2+y^2\leq z^2$. This set has two connected components given by $\pm z>0$. The claim then follows easily from the Cauchy-Schwarz inequality.
\end{proof}

We note the following simple consequence:

\begin{lemma}
\label{ynew}If there are three non-isomorphic stable  (semi)rigid objects of the same phase, then at most one of them is rigid.
\end{lemma}

\begin{proof}
Denote the three objects by $S_1, S_2, S_3$ and their Mukai vectors by $v_i = v(S_i)$. The (semi)rigid assumption gives $(v_i,v_i)\le 0$. As in the proof of Lemma \ref{obv} we have $(v_i,v_j)\ge 0$ for $i\neq j$. We can assume that one of the objects, say $S_1$, is rigid, so that $(v_1,v_1)< 0$. Then by   Lemma \ref{z}, we have  $(v_1,v_i)>0$ for $i=2,3$ and we conclude that $v_1$ lies in one component, say $\CC^+$, and  that $v_2,v_3$ lie in the opposite component, $\CC^-$. Suppose now that one of the objects $S_2$ or $S_3$ is also rigid. Then we can apply the same argument  and conclude that $v_2$ and $v_3$ lie in opposite cones. This gives  a contradiction. 
\end{proof}

This leads to the following useful description of semistable (semi)rigid objects.
 
\begin{prop} \label{cases}
 Suppose that $F$ is a semistable (semi)rigid object. Then exactly one of the following holds:
\begin{itemize}
\item[(a)] $F\isom S^{\oplus k}$ with $S$ a stable spherical object and $k\geq 1$;
\item[(b)] $F$ is stable and semirigid;
\item[(c)] exactly 2 stable objects $S_1,S_2$ occur as stable factors of $F$, and their Mukai vectors $v(S_i)$ are linearly independent in $\N(X)$.
\end{itemize}
\end{prop}

\begin{proof}
Lemma \ref{stab} implies that $F$ has at most one semirigid stable factor, the others being rigid.
Therefore Lemma \ref{ynew} shows that, up to isomorphism, there  are at most two objects  occurring as stable factors of $F$.    If only one occurs then we are in cases (a) or (b) according to whether it is rigid or semirigid. If two occur then  Lemma \ref{obv} shows that  we are in case (c).  \end{proof}

Note that in the situation of Prop.~\ref{cases} the object $F$ is quasistable in cases (a) and (b), but not in  case (c).

\subsection*{Codimension one walls}
\label{cod}

Suppose that \[\sigma_0=(Z_0,\P_0)\in \Stab(X)\] is a stability condition, and that $F\in
\P_0(\phi)$ is a (semi)rigid semistable object which is not quasistable. Prop.~\ref{cases} shows that $F$ has exactly two stable factors $S_1,S_2$ up to isomorphism, whose Mukai vectors $v(S_i)$ are linearly independent. Lemma \ref{reduce} shows that to understand stability of $F$ near $\sigma_0$ it is enough to consider stability functions on the abelian subcategory    $\A\subset \P_0(\phi)$ consisting of those objects  all of whose stable factors are isomorphic to one of the $S_i$. Note that the inclusion $\A\subset \D(X)$ induces an injective  group homomorphism
$\Z^2 \cong K(\A)\into \N(X)$, so we can identify $K(\A)$ with the sublattice of $\N(X)$ spanned by the $v(S_i)$.
For future reference we make the following observation:

\begin{lemma}
\label{obs}
Suppose  that  $Z$ is a stability function on $\A$ and let $\Theta\subset (0,1]$ be the set of phases of $Z$-stable objects of $\A$. Suppose that $F\in \A$ is $Z$-stable and rigid. Then $\phi(F)\in \Theta$ is not an accumulation point.
\end{lemma}

\begin{proof}
 We can assume that $\Im Z(S_1)/Z(S_2)\neq 0$ since otherwise  $\Theta$ consists of a single point. Then $Z$ induces an isomorphism of real vector spaces $K(\A)\tensor \R\isom \C$, and we can think of $\Theta$ as a subset of
\[S^1=(K(\A)\tensor \R \setminus \{0\})/\R_{>0}.\] Suppose we have stable objects $F_n$ whose Mukai vectors $v_n\in K(\A)$  define points $[v_n]\in S^1$ which converge to the point $[v]$ defined by $v=v(F)$. This means that there are positive real numbers $\lambda_n$ such that $\lambda_n\cdot v_n \to v$. In particular, $\lambda_n^2 \cdot (v_n,v_n)\to (v,v)$. Since $F$ is rigid we have $(v,v)=-2$, so omitting finitely many terms of the sequence we can assume that $(v_n,v_n)<0$ for all $n$. But since the objects $F_n$ are stable this implies that $(v_n,v_n)=-2$ for all $n$. Therefore, the sequence $\lambda_n$ converges to 1, and hence the sequence $v_n$ converges to $v$. But since the vectors $v_n$ lie in the integral lattice this means that the sequence must be eventually constant.
\end{proof}

Let us now consider the abstract situation where $\A$ is a finite length abelian category with two simple objects $S_1$, $S_2$, and $Z\colon K(\A) \to \C$ is a stability function. We
note the following trivial statement.

\begin{lemma}
\label{wally}
Suppose that \[\Im Z(S_1)/Z(S_2)\neq 0.\]
 Then any $Z$-semistable object in $\A$ is automatically  $Z$-quasistable.
\end{lemma}

\begin{proof}
Since $Z$ induces an isomorphism of real vector spaces $K(\A)\tensor \R \isom \C$, two objects have the same phase precisely if their classes lie on a ray in $K(\A)$.
\end{proof}

Later on we shall need the following more difficult  result.

\begin{lemma}
\label{hein}
Suppose  $F_1,F_2\in \A$ are $Z$-stable and satisfy
\[ \Hom_X(F_1,F_2)=0=\Hom_X(F_2,F_1).\]
Let $\Theta\subset (0,1]$ be the set of phases of stable objects of $\A$, and assume that at least one of the phases $\phi(F_i)$ is not an accumulation point of $\Theta$. Then, possibly after reordering the $F_i$, we have $F_i\isom S_i$.
\end{lemma}

\begin{proof}
The pair $(Z,\A)$ induces a stability condition $(Z,\P)$ on the bounded derived category $\D=\D^b(\A)$ in the usual way. Set $\phi_i=\phi(F_i)\in (0,1]$, and reorder the objects $F_i$ so that $\phi_1\leq \phi_2$.
We  treat first the case when $\phi_1$ is not an accumulation point of $\Theta$.

Consider the heart $\CC=\P([\phi_1, \phi_1+1))\subset \D$. 
Note that   $F_1,F_2\in \CC$, and $F_1$ is a simple object of $\CC$. The assumption that $\phi_1$ is not an accumulation point of $\Theta$ implies that  $\CC=\P([\phi_1,\phi_1+1-\epsilon])$ for some $\epsilon>0$.
Thus, the central charges $Z(E)$ for $E \in \CC$ are contained in a strictly convex sector of the
complex plane. Also note that $K(\CC)=K(\A) = \Z^{\oplus 2}$, and thus the set of central charges
$Z(E)$ for $E \in \CC$ is discrete in this sector. It follows that $\CC$ is of finite length, and
since $K(\CC)$ has rank 2 that there is exactly one other simple object in $\CC$ up to isomorphism, say $T$.
The effective cone in $K(\CC)$ is then generated by the classes of the simple objects $F_1$ and $T$,   and it follows that  $\phi(F_1)\leq \phi(F_2)\leq \phi(T)$. 

 The assumption  $\Hom_\A(F_1,F_2)=0$  shows that  $F_1$ is not a subobject of $F_2$ in $\CC$. It
follows that $T$ is a subobject of $F_2$, and in particular  there is a nonzero map $T\to F_2$. But
since $F_2$ is $Z$-stable this is only possible if   $F_2=T$. Then $\CC=\P([\phi_1,\phi_2])$ is a
subcategory of $\A$. But since $\A$ and $\CC$ are both hearts in $\D$, this implies that $\A=\CC$,
and therefore   $F_1$ and $F_2$ are the two simple objects of  $\A$ up to isomorphism.

If we instead assume that $\phi_2$ is not an accumulation point of $\Theta$, then we can consider
the finite length heart  $\CC=\P((\phi_2-1,\phi_2])\subset \D$ in which  $F_2$ is simple, and apply a similar argument.
\end{proof}


\subsection*{Width function}
Recall our fixed (semi)rigid object $E\in \D(X)$. Given a stability condition $\sigma \in \Stab(X)$, 
we define the \emph{width} of $E$ by
\[w_E(\sigma)=\phi_\sigma^+(E)-\phi_\sigma^-(E)\in \R_{\geq 0},\]
which we view as a continuous function
\[w\colon \Stab(X)\to \R_{\geq 0}.\]
It is evidently invariant under the $\C$-action.

We denote by $E_\pm = E_\pm(\sigma)$ the HN factors of $E$ with maximal and minimal phase $\phi_\pm$. We denote by
$n=\lfloor w_E(\sigma)\rfloor\geq 0$, and define $A_\pm$ by $A_+=E_+$ and $A_-=E_-[n]$. Note  that $A_\pm$ are semistable and 
\[0\leq \phi(A_+)-\phi(A_-) <1.\]
We shall repeatedly use the following result.

\begin{lemma}
\label{x}
Assume that $w_E(\sigma)>0$. Then
\begin{itemize}
\item[(a)]
the objects $A_\pm$ are both either rigid or semirigid, and at most one of them is semirigid, and\smallskip
\item[(b)]
$\Hom^i_X(A_-,A_+)=0$ unless $i\in \{1,2\}$;
 \end{itemize}
if we assume in addition that $w_E(\sigma)\notin \Z$, we also have
\begin{itemize}
\item[(c)] $\Hom^i_X(A_-,A_+)=0$ unless $i=1$, and
\smallskip
\item[(d)] $(v(A_+),v(A_-))>0$. 
\end{itemize}
\end{lemma}

\begin{proof}
Applying Lemma \ref{stan} repeatedly to the HN filtration of $E$ gives  (a). 
The objects $A_\pm$ lie in the heart $\A=\P((\phi_+-1,\phi_+])$ on $\D(X)$ and hence, using Serre duality,  satisfy  \begin{equation}
\label{apm}
\Hom_X^i(A_\pm,A_\pm)=0\text{ unless } i\in\{0,1,2\}.\end{equation}
 For (b) we must show that $\Hom_X(A_-,A_+)=0$. Note that taking cohomology with respect to the above heart $\A$ we have $H^i_\A(E)=0$ unless $0\leq i \leq n$. Moreover, there is an epimorphism $H_\A^{n}(E)\to A_-$ and a monomorphism $A_+\to H_\A^0(E)$. Suppose there is a nonzero map $f\colon A_-\to A_+$. Using the spectral sequence
\begin{equation}
\label{spec}
E_2^{p,q}=\bigoplus_{i\in \Z}\Hom^p_X(H^i_\A(E),H_\A^{i+q}(E)) \implies \Hom^{p+q}_X(E,E)\end{equation}
it follows that there is a nonzero map $E\to E[-n]$ which if $n>0$ contradicts  the fact that $E$ is semirigid. In the case $n=0$ we have that $E\in \A$ and an epimorphism $g\colon E\to A_-$ and a monomorphism $h\colon A_+\to E$. Then $h\circ f\circ g$ is a nonzero map $E\to E$ which by the assumption $\Hom_X(E,E)=\C$ must be a multiple of the identity. It follows that $E\isom A_+\isom A_-$ which contradicts the assumption that  $w_E(\sigma)>0$.

For (c), note first that since $\phi(A_+)>\phi(A_-)$, and since the objects $A_\pm$ are semistable, there are no nonzero maps $A_+\to A_-$; then apply Serre duality.
The inequality of part (d) then follows by Riemann-Roch. Equality   is impossible, by Lemma \ref{z}, since at least one of $A_\pm$ is rigid.
\end{proof}


\subsection*{$(\pm)$-walls}

We say that a stability condition $\sigma\in \Stab(X)$  is $(+)$-generic if $E_+$  is quasistable, and similarly $(-)$-generic if $E_-$ is quasistable.

\begin{lemma}
\label{wall1}
The subset of $(+)$-generic stability conditions is  the complement of a real, closed  submanifold
$W_+\subset \Stab(X)$ of codimension 1. The object $E_+$ is locally constant on $W_+$, as well as on
the complement of $W_+$. Similarly for $(-)$-generic stability conditions and $E_-$.
\end{lemma}

\begin{proof}The first claim is that being $(+)$-generic is an open condition. Indeed,  the first step in the HN filtration of $E$  in a stability condition $\sigma_0=(Z_0,\P_0)$ is a triangle \begin{equation}
\label{tri}E_+\to E\to F\end{equation}
with $E_+\in \P_0(\phi_+)$ and $F\in \P_0(<\phi_+)$. If $E_+$ is moreover quasistable,
Prop.~\ref{open} shows that $E_+$ remains semistable in a neighbourhood of $\sigma_0$. It follows
that \eqref{tri} remains the first step in the HN filtration of $E$, so that  the object $E_+$ is
locally constant, and the claim follows then from Prop.~\ref{open}. 

 Suppose now that $\sigma_0\in \Stab(X)$ is not $(+)$-generic. By Prop.~\ref{cases} there are then
exactly two stable factors $S_1,S_2$ of $E_+$, whose Mukai vectors $v(S_i)$ are linearly independent. These objects generate a Serre subcategory $\A\subset \P_0(\phi_+)$ as in Section 2. Lemma \ref{reduce} and the argument above show that for stability conditions $\sigma=(Z,\P)$ in a neighbourhood of $\sigma_0$, the maximal HN factor of $E$  in $\sigma$ is  precisely the maximal HN factor of the object $E_+\in \A$ with respect to the  stability function $Z$. The locus of non $(+)$-generic stability conditions is therefore  the set of points satisfying
\begin{equation}
\label{bl}
\Im Z(S_1)/Z(S_2) =0.\end{equation}
Indeed, when this condition is satisfied, $E_+$  itself is semistable but not quasistable. On the
other hand, when \eqref{bl} does not hold,  Lemma \ref{wally}  shows that the maximal HN factor is automatically quasistable.
\end{proof}

We refer to the connected components of the submanifold $W_+$ as $(+)$-walls.   Similarly for  $(-)$-walls.
A connected component of the complement
\begin{equation}\label{comp}\Stab(X)\setminus (W_+\cup W_-)\end{equation}
will be called a chamber. 
Lemma \ref{wall1} shows that both objects $E_\pm$ are constant on every chamber.
In particular,  the function $w_E$ is smooth on every chamber. In fact, it is also smooth on the closure of
each chamber: as $\sigma$ reaches a $(+)$-wall, the object $E_+$ may change, but the
phases $\phi(E_+)$ of the two objects agree on the wall; thus $\phi(E_+)$ extends to a smooth function on the closure
of the chamber.

\begin{remarks}
\begin{itemize}
\item[(a)]The union
of the submanifolds $W_+$ and $W_-$ need not be a submanifold since $(+)$ and $(-)$ walls can intersect each other. 

\item[(b)]
The object $A_-$ need not be  locally-constant on the complement \eqref{comp} since its definition involves the function $\lfloor w_E\rfloor$ which is discontinuous at points of $w^{-1}(\Z)$.
\end{itemize}
\end{remarks}


\subsection*{Integral walls}

Consider now the subset
\[W_\Z=\{\sigma\in \Stab(X)\colon w_E(\sigma)\in \Z \text{ and $E$ is not $\sigma$-quasistable}\}.\]
We call the connected components of $W_\Z$ integral walls.

\begin{remark}
The condition that $E$ is not $\sigma$-quasistable is essential for the following result to hold.
For example, when $E=\O_x$ is a skyscraper sheaf there is an open region $U$ where $E$ is stable
(this coincides with the subset of geometric stability conditions, see Lemma \ref{ble}): we do not
want $W_\Z$ to contain the closure of  this subset, but rather its boundary  (see also Prop.~\ref{zerocase} below). 
\end{remark}

\begin{lemma}\label{wall2}The subset $W_\Z$ is a real, closed  submanifold of $\Stab(X)$ of codimension 1.
\end{lemma}

\begin{proof}
Since $w$ is continuous, the subset $w^{-1}(\Z)$ is closed. The fact that $W_\Z$ is closed then
follows from Prop.~\ref{open}. If $\sigma\in \Stab(X)$ satisfies  $w_E(\sigma)=0$ then $\sigma$ lies on $W_\Z$ precisely if $\sigma$ is not $(\pm)$-generic, so for these points the result follows from Lemma \ref{wall1}. 
Thus we can work in a neighbourhood of a point $\sigma_0\in W_\Z$ for which  $w_E(\sigma_0)> 0$.

Consider the stable factors $S_i$ of $A_\pm$. By Lemma \ref{stab} they are all
(semi)rigid, and at most one is semirigid. Lemma \ref{ynew} shows that there at most two of them. But if there is only one  then there is a nonzero map $A_-\to A_+$ contradicting Lemma \ref{x} (b). Hence the objects $A_\pm$ have exactly 2 stable factors $S_1,S_2$ between them, and by Lemma \ref{obv} the  Mukai vectors $v(S_i)$ are linearly independent.

We claim that in a neighbourhood of $\sigma_0$ the closed subset $W_\Z$  is cut out by the equation
\begin{equation}
\label{again}\Im Z(S_1)/Z(S_2) =0.\end{equation}
Certainly, if this condition is satisfied, the stable factors of $A_{\pm}$ remain stable (by Prop. ~\ref{open}) and of equal phases; hence the objects $E_\pm$ remain semistable, and continue to be the extremal HN factors of $E$. Thus $w_E(\sigma)$ remains integral.

 For the converse, we apply Lemma \ref{reduce} to conclude that for any $\sigma=(Z,\P)$  in a neighbourhood of $\sigma_0$, the extremal HN factors of $E$ (up to shift) are just given by the extremal HN factors of the objects $A_\pm\in \A$ with respect to the stability function $Z$. Call these objects $C_\pm$ and suppose that $\phi(C_+)=\phi(C_-)$. By Lemma \ref{obv} the Mukai vectors of the distinct stable factors of  $C_\pm$  are linearly independent in $\N(X)$. There must be more than one of them by Lemma \ref{x}(b). Since $Z$ maps these different stable factors onto a ray, condition \eqref{again} must hold.
\end{proof}

\begin{remarks} \label{rem:pmwallsintersect}
\begin{itemize}
\item[(a)]
It follows from the local descriptions given in Lemmas \ref{wall1} and \ref{wall2} that if a $(\pm)$-wall $W_1$ intersects an integral wall  $W_2$ then in fact $W_1=W_2$ is simultaneously a $(+)$-wall, a $(-)$-wall, and an integral wall.\smallskip

\item[(b)]
It is easy to check that if a $(+)$-wall  and
a $(-)$-wall   coincide then this wall is also an integral wall. 

\item[(c)] The statements of Lemmas \ref{wall1} and \ref{wall2} continue to hold if we replace the manifold $\Stab(X)$ by the complex submanifold $\Stab^*_\red(X)$ of reduced stability conditions. The point is that the walls, which are cut-out by equations of the form $(\Omega,v_1)/(\Omega,v_2)\in \R$, intersect the complex submanifold given by $(\Omega,\Omega)=0$ transversely. 
\end{itemize}
\end{remarks}

\subsection*{No local minima}

Using the action of the universal cover of  $\operatorname{GL^+_2}(\R)$ on $\Stab(X)$, it is easy to
see  that if the width function $w_E$ had a local minimum  $ \sigma_0\in \Stab(X)$, then this would have to satisfy $w_E({\sigma_0})\in \Z$. The following crucial result then shows that in fact the  function $w_E$ has no positive local minima on $\Stab(X)$.

\begin{prop} \label{lem:integerwalls}
Let $\sigma_0\in W_\Z\subset \Stab(X)$  satisfy $w_E(\sigma_0)=n\in \Z_{>0}$. Locally near $\sigma_0$ the
submanifold $W_\Z$ splits $\Stab(X)$ into two connected components, with
$w_E(\sigma)<n$ in one component and $w_E(\sigma)>n$ in the other. 
\end{prop}

\begin{proof}
As in the proof of Lemma \ref{wall2}, the objects $A_\pm$  have two stable factors $S_1,S_2$ between
them; we again consider the finite length abelian subcategory $\A\subset \P_0(\phi)$ generated by
$S_1$ and $S_2$. Considering the Jordan-H{\"o}lder filtration of $A_+$ and relabelling the objects
$S_i$ if necessary, we can assume that there is a monomorphism  $S_1\to A_+$ in  $\A$. It follows from Lemma
\ref{x}(b) that $\Hom_X(A_-,S_1)=0$, and hence  there must  be an epimorphism  $A_-\to S_2$ in $\A$. Using Lemma \ref{x}(b) again,
this in turn implies that $\Hom_X(S_2, A_+) = 0$.

Locally near $\sigma_0$ the submanifold $W_\Z$  is cut out by the equation $Z(S_1)/Z(S_2)\in \R_{>0}$.
Note that there are no other relevant walls near $\sigma_0$, since any $(\pm)$-wall which intersects
$W_\Z$ coincides with a component of $W_\Z$. Hence it makes sense to speak of the new objects
$A_\pm$ for stability conditions $\sigma=(Z,\P)$ on either side of the wall. By Lemma \ref{reduce}, the new $A_+$ is the maximal HN factor of $A_+$ with respect to the slope function $Z$, and, up to shift,  the new $A_-$ is the minimal HN factor of $A_-$ with respect to $Z$. 
 
On one side of the wall $\phi(S_1)>\phi(S_2)$. Then  the new $A_+$ is in the subcategory generated
by $S_1$, and the new $A_-$ is in the subcategory generated by $S_2$. The width $w_E(\sigma)$ has increased to $n+\phi(S_1)-\phi(S_2)$.

On the other side of the wall $\phi(S_1)<\phi(S_2)$.  Write $C_\pm$ for the new objects $A_\pm$;
recall that $C_+ = E_+$ and $C_- = E_-[-n]$, where $E_\pm$ are the first and last factor of the new
Harder-Narasimhan filtration of $E$, and $n = \lfloor w_E(\sigma) \rfloor$ is determined by the
width.
Suppose that  the width has also increased on this side. Then  $n=\lfloor w_E\rfloor$ is unchanged,
and so $C_-$  is precisely the minimal HN factor of $A_-$ (rather than its shift).

 By Lemma \ref{wally},  the object $C_+$ is quasistable. Prop. \ref{cases} then shows that it has a single stable factor, call it $T_+$. Similarly, $C_-$ has a single stable factor $T_-$. By Lemma \ref{x} we have  $\Hom_X^k(C_+,C_-)=0$ unless $k=1$, and it follows that \[\Hom_\A(T_-,T_+)=0=\Hom_\A(T_+,T_-).\]
 Applying Lemma \ref{obs} and Lemma \ref{hein} we conclude that  $\{T_-,T_+\}=\{S_1,S_2\}$.  The assumption that the width has increased implies that in fact $T_-=S_1$ and $T_+=S_2$. We thus get a chain of inclusions $S_2=T_+ \subset C_+ \subset A_+$ in $\A$, in contradiction to 
$\Hom_X(S_2, A_+) = 0$ observed above.
\end{proof}

A similar result holds at points of $W_\Z$ of width zero.

\begin{prop} \label{zerocase}
Let $\sigma_0\in W_\Z\subset \Stab(X)$  satisfy $w_E(\sigma_0)=0$. Locally near $\sigma_0$ the
submanifold $W_\Z$ splits $\Stab(X)$ into two connected components. In one $E$ is quasistable, and in the other $w_E(\sigma)>0$. 
\end{prop}

\begin{proof}
The proof is very similar to that of Prop.~\ref{lem:integerwalls}, and we just indicate the
necessary modifications. In the first paragraph, note that $A_+=A_-=E$. Considering the
Jordan-H{\"o}lder filtration of $E$ we can assume that there is a monomorphism $S_1\to E$ as before.
By the assumption that $\sigma_0 \in W_\Z$, this is not an isomorphism, as $E$ cannot be stable.
If there were a nonzero map $E\to S_1$ this would contradict the assumption that $\Hom_X(E,E)=\C$.
It follows
that there is an epimorphism $E\to S_2$ in $\A$. The rest of the proof of Prop.~\ref{lem:integerwalls} applies without change and shows that on one side of the wall $w_E(\sigma)=0$ (i.e. $E$ is semistable), and on the other $w_E(\sigma)>0$. Finally, note that on the first side we must in fact have $E$
quasistable, since if it is not,  $\sigma$ lies on $W_\Z$.
\end{proof}


\section{Flow}
\label{sec:flow}

Let us fix a complex projective K3 surface $X$ and an object $E\in \D(X)$ with assumptions as in Section 3. In particular, the surface $X$ has Picard rank $\rho(X)=1$ and the object $E$ is (semi)rigid and satisfies $\Hom_X(E,E)=\C$. The aim of this section is to construct a flow on the space $\Stab^*_\red(X)$ that decreases the width $w_E$ to the nearest integer
$\lfloor w_E \rfloor$.

\subsection*{Construction} \label{sect:construct}
Let $\sigma=(Z,\P)\in \Stab^*_{\red}(X)$ be a reduced stability condition, and write  $Z(E)=(\Omega,v(E))$ for
some vector $\Omega\in  \N(X)\tensor \C$. Recall that $\Omega \in \Q^+_0(X)$, so in particular
\[(\Omega,\Omega)=0, \quad (\Omega,\bar{\Omega})=2d>0,\]
and the orthogonal to the 2-plane in $\N(X)\tensor \R$ spanned by the real and imaginary parts of  $\Omega$ is a negative definite line.
Let $\Theta=\Theta(\sigma)\in \N(X)\tensor \R$ be the unique  vector satisfying
\begin{equation} \label{eq:defTheta}
(\Theta,\Omega)=0, \quad -(\Theta,\Theta)=d=\half(\Omega,\bar{\Omega}), \quad \Theta\in \CC^+.
\end{equation}

For any stability condition $\sigma$, we define a sign $\epsilon = \epsilon(\sigma)\in \{\pm 1\}$ by the condition
that $v(E_+)=v(A_+)$ lies in $\CC^{\epsilon(\sigma)}$.
\begin{lemma} \label{lem:epsilonlocalconstant}
The sign $\epsilon(\sigma)$ is locally constant on the complement of $W_\Z$.
\end{lemma}

\begin{proof}
By Lemma \ref{wall1}, the object $A_+$ is locally constant  on the
complement of the set of $(+)$-walls. By definition, it follows that $\epsilon$ is
also locally constant on this complement.

Similarly, $E_-$ is locally constant on the complement of the $(-)$-walls and 
if we restrict further to the complement of $W_\Z$, the same holds for $A_-$. But due
to Lemma \ref{x} (d) and Lemma \ref{z}, on the complement of $W_\Z$ the objects $A_\pm$ lie in different connected components of the cone $\CC$; thus the object $A_-$ also determines  the sign $\epsilon(\sigma)$.

Now recall from Remark  \ref{rem:pmwallsintersect} (b) that a $(+)$-wall can only coincide with
a $(-)$-wall if they are contained in $W_\Z$. Thus we have proved that $\epsilon(\sigma)$ is locally
constant on $\Stab(X) \setminus W_\Z$.
\end{proof}

Define a complex number of unit modulus by
\begin{equation} \label{eq:defzeta}
\zeta=\zeta(\sigma)=i\cdot \exp \frac{i\pi}{2}\big(\phi(A_+)+\phi(A_-)\big).\end{equation}
Finally, define a nonzero vector
\begin{equation} \label{eq:defv}
v=v(\sigma)=\epsilon(\sigma) \cdot  
\zeta(\sigma)\cdot  \Theta(\sigma) \in \N(X)\tensor \C.\end{equation}

\begin{lemma} \label{lemma:flowlocal}
The flow
\[\dbydt{} \Omega = v(\sigma)\]
of the vector field $v(\sigma)$ exists locally uniquely on
$\Stab^*_\red(X) \setminus w_E^{-1}(\Z)$,
the space of reduced stability conditions of non-integral width.
It preserves the positive real number $2d=(\Omega, \overline{\Omega})$. 
\end{lemma}
In other words, for any point in $\Stab^*_\red(X) \setminus w_E^{-1}(\Z)$, there exists an open
neighbourhood $U$, $\epsilon >0$, and a continuous map $U \times [0, \epsilon) \to \Stab^*_\red(X)$
solving the differential equation given by the vector field $v(\sigma)$.
\begin{proof}
The vector $v(\sigma)$ varies continuously on $\Stab^*_\red(X)$ by the above Lemma. Since the set of $(\pm)$-walls is
locally finite, the resulting vector field is Lipschitz continuous on every compact subset; by the 
Picard-Lindel\"of Theorem, the flow then exists locally and is unique.
From $(\Omega, \Theta) = 0 = (\overline{\Omega}, \Theta)$ one obtains
\[\dbydt{} (\Omega, \Omega) = \dbydt{} (\Omega, \overline{\Omega}) = 0.\] Thus
the condition of being reduced is preserved, and $(\Omega, \overline{\Omega})$ is constant.
\end{proof}

\subsection*{Flow decreases width}
Simple sign observations show that the flow defined in the last subsection moves $Z(A_\pm)$ in the direction
$\mp\zeta(\sigma)$ and hence decreases the width, see Figure \ref{fig:flow}.
To make this observation precise, we first point out that since 
$w_E$ is smooth on the closure of each chamber, the function
$w_E(\sigma(t))$ will be piecewise differentiable. Thus we can define $\dbydt {w_E}(\sigma)(t)$ at time $t$
to be the derivative of $w_E$ restricted to the interval $[t, t+\epsilon)$.

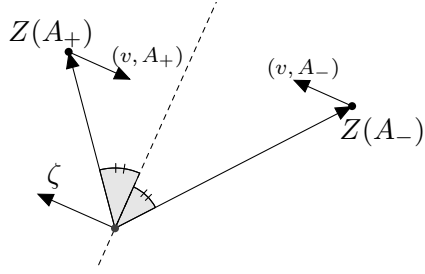
\begin{figure}
\definecolor{uququq}{rgb}{0.25,0.25,0.25}
\begin{tikzpicture}[line cap=round,line join=round,>=triangle 45,x=1.0cm,y=1.0cm]
\clip(-1.7,-0.5) rectangle (4.2,3);
\draw [shift={(0,0)},fill=black,fill opacity=0.1] (0,0) -- (27.44:0.6) arc (27.44:66.04:0.6) --
cycle;
\draw [shift={(0,0)},fill=black,fill opacity=0.1] (0,0) -- (66.04:0.8) arc (66.04:104.64:0.8) --
cycle;
\draw [line width=0.4pt,dash pattern=on 2pt off 2pt,domain=-1.7:4.2] plot(\x,{(-0-3.6*\x)/-1.6});
\draw [->] (0,0) -- (3.12,1.62);
\draw [->] (0,0) -- (-0.61,2.34);
\draw [->,line width=0.4pt] (0,0) -- (-1.03,0.46);
\draw [shift={(0,0)}] (27.44:0.6) arc (27.44:66.04:0.6);
\draw(0.34,0.42) -- (0.42,0.51);
\draw(0.4,0.37) -- (0.48,0.45);
\draw [shift={(0,0)}] (66.04:0.8) arc (66.04:104.64:0.8);
\draw(0.01,0.74) -- (0.01,0.86);
\draw(0.11,0.73) -- (0.13,0.85);
\draw [->] (-0.61,2.34) -- (0.21,1.97);
\draw [->] (3.12,1.62) -- (2.34,1.97);
\fill [color=uququq] (0,0) circle (1.5pt);
\fill [color=black] (3.12,1.62) circle (1.5pt);
\draw[color=black] (3.52,1.28) node {$Z(A_-)$};
\fill [color=black] (-0.61,2.34) circle (1.5pt);
\draw[color=black] (-0.84,2.58) node {$Z(A_+)$};
\draw[color=black] (-0.8,0.7) node {$\zeta$};
\begin{scriptsize}
\draw[color=black] (0.43,2.27) node {$(v, A_+)$};
\draw[color=black] (2.48,2.12) node {$(v, A_-)$};
\end{scriptsize}
\end{tikzpicture}
\caption{The flow and $Z(A_\pm)$}
\label{fig:flow}
\end{figure}

\begin{lemma} \label{lem:decreasingwidth}
Under the flow of $v(\sigma)$, the functions $\phi(A^+)$ and $\phi(A^-)$ are decreasing and
increasing, respectively, and the derivative
$\dbydt{} w_E(\sigma(t))$ is negative. Moreover, setting $\theta=w_E(\sigma)-\lfloor w_E(\sigma)\rfloor$, one has
\[-\dbydt{} w_E(\sigma) \geq \frac{2}{\pi} \cos \left(\frac{\pi \theta}2\right) > 0.\]
\end{lemma}

\begin{proof}
Note that
\[\epsilon(\sigma)\cdot  (\Theta,A_+)<0, \quad \epsilon(\sigma)\cdot  (\Theta,A_-)>0.\]
Indeed, since $\Theta\in \CC^+$ we have $(\Theta,A_+)\in \epsilon(\sigma) \cdot \R_{<0}$, and the
objects $A_\pm$ have classes in opposite cones by Lemma \ref{x}. Thus the flow has the effect of
adding negative multiples of the vector $\zeta$ to $Z(A_+)$ and positive multiples to $Z(A_-)$. It
is then clear that this decreases $\phi(A_+)$ and increases $\phi(A_-)$.

Writing $\Omega=X+iY$ and $\Theta=W$, the vectors $(W,X,Y)$ form an orthogonal basis for $\N(X)\tensor \R$ such that
\[(X,X)=(Y,Y)=-(W,W)=d>0.\]
It follows that for any vector $v\in \N(X)\tensor \R$ one has
\[|(\Omega,v)|^2 -(\Theta,v)^2 = d(v,v).\]
In particular, if $v=v(E)$ is the Mukai vector of a (semi)rigid object, then $(v,v)\leq 0$
gives
\begin{equation} \label{ineq}
|(\Theta,v)|\geq |Z(E)|.
\end{equation}

\definecolor{uququq}{rgb}{0.25,0.25,0.25}
\begin{figure}
\begin{tikzpicture}[line cap=round,line join=round,>=triangle 45,x=1.0cm,y=1.0cm]
\clip(-2.2,-0.4) rectangle (2.8,3.3);
\draw [shift={(0,0)},fill=black,fill opacity=0.1] (0,0) -- (54.32:0.8) arc (54.32:90:0.8) -- cycle;
\draw [shift={(0,0)},fill=black,fill opacity=0.1] (0,0) -- (90:0.6) arc (90:125.68:0.6) -- cycle;
\draw [->] (0,0) -- (1.68,2.34);
\draw [line width=0.4pt,dash pattern=on 2pt off 2pt] (0,-0.4) -- (0,3.3);
\draw (1.68,2.34)-- (0,2.34);
\draw [->] (0,0) -- (-1.07,1.49);
\draw [line width=0.4pt] (0,1.49)-- (-1.07,1.49);
\draw [->,line width=0.4pt] (0,0) -- (-1,0);
\draw [shift={(0,0)}] (90:0.6) arc (90:125.68:0.6);
\draw(-0.17,0.51) -- (-0.2,0.63);
\fill [color=uququq] (0,0) circle (1.5pt);
\fill [color=black] (1.68,2.34) circle (1.5pt);
\draw[color=black] (2.28,2.6) node {$Z(A_-)$};
\fill [color=uququq] (0,2.34) circle (1.5pt);
\draw[color=uququq] (-0.18,2.25) node {$y_-$};
\draw[color=black] (0.94,2.46) node {$x_-$};
\fill [color=black] (-1.07,1.49) circle (1.5pt);
\draw[color=black] (-1.69,1.62) node {$Z(A_+)$};
\fill [color=uququq] (0,1.49) circle (1.5pt);
\draw[color=uququq] (0.34,1.44) node {$y_+$};
\draw[color=black] (-0.42,1.62) node {$x_+$};
\draw[color=black] (-0.84,0.30) node {$\zeta$};
\draw[color=black] (0.45,0.93) node {$\alpha_-$};
\draw[color=black] (-0.28,0.82) node {$\alpha_+$};
\end{tikzpicture}
\caption{Computing the derivative $v(\sigma).w_E$}
\label{fig:derivative}
\end{figure}
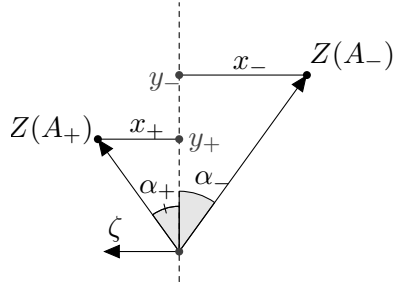

We will prove the inequality at time $t = 0$.
Consider a stability condition $\sigma$ with width $w$. Set $\lfloor w\rfloor =n $ and put
$\theta=w-\lfloor w\rfloor$. Rotating by a fixed scalar $z \in \C$, we can assume that at time $t = 0$,
we have $\zeta=-1$ and $\phi(A_\pm)=(1\pm \theta)/2$.
Set $ Z(A_\pm)=x_\pm + i y_\pm$. As we flow, there are angles $0 \le \alpha_\pm < \frac{\pi}2$ such that
\[ \frac{x_\pm}{y_\pm} =\mp \tan (\alpha_\pm), \quad \frac{y_\pm}{|Z(A_\pm)|}=\cos (\alpha_\pm),\]
see also Figure \ref{fig:derivative}.
At time $t = 0$, we have $\alpha_\pm=\pi\theta/2$.

Since $\zeta(0)$ is real, the derivatives $\dbydt{} y_\pm$ vanish at $t = 0$, and so
\[\mp \sec^2(\alpha_\pm) \left. \dbydt{\alpha_\pm} \right|_{t=0} =\frac 1{\cos(\alpha_\pm)\cdot
|Z(A_\pm)|} \left. \dbydt{x_\pm} \right|_{t=0}.\] 
But $\dbydt{} x_\pm|_{t=0}=\pm |(\Theta,A_\pm)|$, so by the above inequality \eqref{ineq} we get
 \[0\geq -\cos \big(\frac{\pi\theta}{2}\big)\geq \left.\dbydt{\alpha_\pm}\right |_{t=0} .\]
 Writing $\alpha_++\alpha_-=\pi(w-n)$ gives the result.
\end{proof}

In particular, unless the flow ceases to exist at an earlier point in time, it takes a point with non-integral
width $w$ to a point of width $\lfloor w\rfloor$ in finite time less than
$\frac{\pi}2 \left(\cos(\half \pi\theta)\right)^{-1}$.


\subsection*{Global properties}
We now study the flow $\sigma(t)$ defined above in more detail.
Let $\sigma\in \Stab^*_\red(X)$  be a stability condition with $w_E(\sigma) \in (n, n+1)$, and let $I\subset [0,\infty)$ be the maximal
interval of definition of the flow in $w_E^{-1}(n, n+1)$ starting at $\sigma$.
By Lemma \ref{lem:decreasingwidth}, this interval is necessarily finite. Moreover $I$ must be of the form $I=[0,t_0)$ since the flow can always be extended in the neighbourhood of any given stability condition. 
Thus we have a flow
\begin{equation}
\label{flow}\sigma\colon [0,t_0) \lra\Stab^*_\red(X)\setminus w^{-1}(\Z).\end{equation}
Let $\Omega(t) \in \Q^+(X)$ be the underlying flow of central charges, and let us shorten notation
by writing $\Theta(t) := \Theta(\sigma(t)) \in \CC^+$ and $\zeta(t) = \zeta(\sigma(t)) \in \C$
for the quantities defined above in \eqref{eq:defTheta} and \eqref{eq:defzeta}; in addition, recall
from Lemma \ref{lem:epsilonlocalconstant} that $\epsilon = \epsilon(\sigma(t))$ is constant along
the flow line.

\begin{lemma}
The vector $\Theta(t)$ satisfies $(\Theta(t), \Theta(t)) = -d$ for all $t$, and  obeys the differential equation
\[
\dbydt \Theta(t)= \epsilon \Re \zeta(t) \Re \Omega(t) + \epsilon \Im \zeta(t) \Im \Omega(t).
\]
\end{lemma}
\begin{proof}
The fact that $d$ is constant under the flow was already proved in Lemma \ref{lemma:flowlocal}. Write $\Psi(t)$ for the right-hand side of the above equation. It is sufficient to show that
$\dbydt \Theta(t)$ and $\Psi(t)$ have the same pairing with each vector
in the orthogonal basis $\Theta(t)$, $\Re \Omega(t)$ and $\Im \Omega(t)$ of $\NS(X) \otimes \R$.
This follows from
\[ 
\left( \dbydt{\Theta(t)}, \Theta(t) \right) = \half \dbydt{} \bigl(\Theta(t), \Theta(t)\bigr) = 0
= \bigl(\Psi(t), \Theta(t)\bigr),
\]
and the fact that $\dbydt {}(\Theta(t),\Omega(t))=0$, which implies that
\[\left(\dbydt{\Theta(t)}, \Omega(t)\right) =\left(\Theta(t), \dbydt{\Omega(t)}\right)= \epsilon \zeta(t) d= \bigl(\Psi(t), \Omega(t)\bigr).\]
%
%
\end{proof}

Note that $(\xi, \xi) = -d$ defines the Minkowski model of the hyperbolic plane as a subset $\HH
\subset \CC^+$. Up to rescaling by $\frac 1{\sqrt{d}}$, the standard invariant metric on $\HH$ is induced by the quadratic form on
$\N(X) \otimes \R$. In particular, the vectors $\Re \Omega, \Im
\Omega$ form an orthonormal basis of the tangent space to $\HH$ at $\Theta$.
Since $\abs{\zeta} = 1$, the vector $\Theta(t)$ is moving in $\HH$ with
constant speed. Since $\HH$ is complete, the limit $\lim_{t \to t_0} \Theta(t)$ exists, and $\Theta$
extends to a continuous function on the closed interval $[0, t_0]$. 

It follows that $\Omega(t)$ also extends to a continuous function on $[0, t_0]$, as it is the integral of the
continuous function $\epsilon \cdot\zeta(t)\cdot  \Theta(t)$. Since
$(\Omega(t), \overline{\Omega(t)})$ is constant, we also have $\Omega_0:= \Omega(t_0) \in \Q^+(X)$.
If $\Omega_0$ lies in the subset $\Q_0^+(X) \subset \Q^+(X)$, then by Theorem \ref{bridge}, the path $\Omega(t)$ lifts to a continuous path \[\sigma \colon [0, t_0] \lra \Stab^*_\red(X).\]
 By the maximality of the interval $[0,t_0)$ it follows that $w_E(\sigma(t_0)) = n$, as desired. The only other possibility is $(\Omega_0, \delta) = 0$ for some root
$\delta \in \Delta(X)$; we prove that this cannot happen in the next section.
\smallskip

Later on we shall need the following simple statement about the behaviour of the flow near an integral wall:
\begin{lemma}
\label{lem:extends}
The vector field $v(\sigma)$ on the open set \[U = w_E^{-1}(n, n+1)\subset \Stab^*_\red(X)\]
extends continuously to the closure $\overline U$ of $U$ in
$w_E^{-1}([n, n+1))$, and is transversal to the boundary wall where $w_E = n$.
\end{lemma}
\begin{proof}
While the objects $A_+$ and $A_-$ may jump
on the integral wall $w^{-1}(n)$, their phases extend continuously from $U$,
becoming equal on the wall.
Thus  we can extend equations \eqref{eq:defzeta} and \eqref{eq:defv} continuously to $\overline U$.
By the proof of Lemma \ref{lem:integerwalls}, the equation of the wall is locally given by \[\Im Z(A_+)/Z(A_-) = 0,\] where $A_\pm$ are the objects defined by the stability conditions just off the wall. 
Since at the wall, $\zeta$ becomes orthogonal to the ray spanned by $Z(A_\pm)$, it follows from
the same considerations used in the proof of Lemma \ref{lem:decreasingwidth} that
the derivative of the equation with respect to $v(\sigma)$ does not vanish.
In other words, the vector field is transversal to the wall.
\end{proof}

We will also use the following consequence:
\begin{lemma} \label{lem:flowatwall}
With the same notation as in the previous lemma, let $W \subset w_E^{-1}(n)$ be an integral wall
bordering $U$, and let $\sigma\in  W$. 
Then there is a neighborhood $V \subset \overline{U}$ of $\sigma$ with a homeomorphism
\[
V \to \left(V \cap W\right) \times [0, \epsilon)
\]
that identifies the flow on $V$ with the flow on the right-hand side induced by $-\dbydt{}$ on $[0,\epsilon)$.
\end{lemma}
\begin{proof}
We choose a neighborhood $V'\subset W$ of $\sigma$  small enough such that the inverse flow, associated
to $-v$, exists until time $\epsilon$ for all $\sigma' \in V'$. Since $v$ is transversal to the
wall, and since the flow is locally unique, the flow induces an injective local homeomorphism
\[
V' \times [0, \epsilon) \into \overline{U}.
\]
The result follows by taking $V$ to be the image of this map.
\end{proof}


\section{Avoiding holes}
\label{sec:holes}

We continue to assume that $X$ is a K3 surface of Picard rank $\rho(X) = 1$, and that $E$ is a
(semi)rigid object with $\Hom_X(E, E) = \C$.  In this section we will prove that the flow constructed in
the previous section cannot fall down any of the holes $\Q^+(X)\setminus \Q^+_0(X)$. More precisely,
suppose, as before, that the vector field $v(\sigma)$ gives rise to a  flow \eqref{flow} defined on an interval
$[0,t_0)$.  As we proved in the last section, the underlying flow of central charges $\Omega(t)$
extends to a flow \[\Omega(t)\colon [0,t_0] \lra \Q^+(X).\] In this section we prove that in fact
$\Omega_0=\Omega(t_0)\in  \Q^+_0(X)$.

\begin{lemma} \label{lem:zetacone}
There exists $0< t_1<t_0$ with the following property: for all $t \in [t_1, t_0]$,
the vector $\zeta(t)$ lies in the interior of the convex cone spanned by
$Z(A_+)$ and $-Z(A_-)$ at time $t = t_1$.
\end{lemma}
\begin{proof}
We write $\tilde{\phi}^\pm(t)$ for the phases $\phi(A_\pm)$ as a function of $t$, and recall that
$\tilde{\phi}^+ < \tilde{\phi}^- + 1$. We 
have to show that $t_1$ can be chosen so that for all $t\in [t_1,t_0]$ an appropriate branch of\[\psi(t) := \frac 1\pi \arg \zeta(t)\] lies strictly between
$\tilde{\phi}^+(t_1)$ and
$\tilde{\phi}^-(t_1) + 1$, see also Figure \ref{fig:zetacone}.
By Lemma \ref{lem:decreasingwidth}, the functions $\tilde{\phi}^+(t)$ and $\tilde{\phi}^-(t)$ are bounded monotone decreasing
and increasing, respectively, and thus extend to continuous functions on $[0, t_0]$.
By definition of $\zeta(t)$, we have
$$  \label{psi} \tilde{\phi}^+(t) < \psi(t) = \half\big( {\tilde{\phi}^+(t) + \tilde{\phi}^-(t) + 1}\big)
< \tilde{\phi}^-(t) + 1$$
for all $ 0 \le t \le t_0$. The claim then follows by continuity.
\end{proof}

\begin{figure}
\definecolor{qqwuqq}{rgb}{0,0.39,0}
\definecolor{uququq}{rgb}{0.25,0.25,0.25}
\begin{tikzpicture}[line cap=round,line join=round,>=triangle 45,x=1.0cm,y=1.0cm]
\clip(-3,-1.5) rectangle (2.7,2.7);
\draw [shift={(0,0)},color=qqwuqq,fill=qqwuqq,fill opacity=0.1] (0,0) -- (158.84:1) arc
(158.84:218.66:1) -- cycle;
\draw [->] (0,0) -- (-2.48,0.96);
\draw [->] (0,0) -- (2,1.6);
\draw [line width=0.4pt,dash pattern=on 2pt off 2pt,domain=-3:2.7] plot(\x,{(-0-1.6*\x)/-2});
\draw [->] (0,0) -- (-2.02,-0.46);
\fill [color=uququq] (0,0) circle (1.5pt);
\fill [color=black] (-2.48,0.96) circle (1.5pt);
\draw[color=black] (-2.38,1.26) node {$Z(A_+)$};
\fill [color=black] (2,1.6) circle (1.5pt);
\draw[color=black] (1.62,1.90) node {$Z(A_-)$};
\fill [color=black] (-2.02,-0.46) circle (1.5pt);
\draw[color=black] (-1.75,-0.18) node {$\zeta(t)$};
\end{tikzpicture}
\caption{Restraining $\zeta(t)$ as $t$ approaches a boundary point} \label{fig:zetacone}
\end{figure}
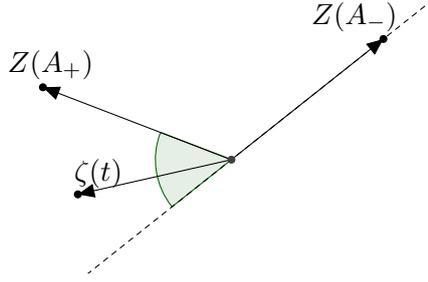

 Assume for a contradiction that $(\Omega_0, \delta) = 0$ for some root
$\delta \in \Delta(X)$. Note that $\delta$ is then uniquely defined up to sign.
 Write $Z_0$ for the central charge
$Z(\blank) = (\Omega_0, \blank)$. We will need to consider small neighborhoods of $Z_0$;
to this end, we choose a norm $\| \cdot \|$ on $\N(X) \otimes \R$, and also write $\|\cdot\|$ to
denote the induced operator norm on $\Hom_\Z(\N(X), \C)$. We first observe that such $Z$ satisfies
the support property for all classes other than $\delta$:
\begin{lemma} \label{lem:supppropertynotdelta}
There exist constants $\epsilon, K \in \R_{>0}$ such that for all $v \in \N(X)$ with $(v, v) \ge -2$ 
and $v \neq \pm \delta$, and all $Z \in \Hom(\N(X), \C)$ we have
\begin{equation} \label{eq:supportnotdelta}
\|Z - Z_0 \| < \epsilon \quad \Longrightarrow \quad \abs{Z(v)} \ge K \| v \|.
\end{equation}
\end{lemma}
\begin{proof}
The proof of the support property for $\Stab(X)$, \cite[Lemma
8.1]{Bridgeland:K3}, applies identically.
\end{proof}

\begin{lemma} \label{lem:small}
Consider a sufficiently small open neighbourhood $Z_0\in V\subset \P^+(X)$, and let $U\subset \St(X)$ be a
connected component of $\pi^{-1}(V\cap \P^+_0(X))$.  Then there is an object $S$ of class $v(S) = \delta$ that
is stable for all stability conditions in $U$.
\end{lemma}

\begin{proof}
We first claim that if $V$ is sufficiently small, 
then its preimage does not intersect any walls at which a stable
object $S$ of class $\delta$  becomes strictly semistable. Note that any such object $S$ is necessarily rigid.

Let $A$ be a stable factor of $S$ on such a wall. It follows that $\abs{Z(S)}  >
\abs{Z(A)}$. By Lemma
\ref{stab}(a), the object $A$ is rigid, so $v=v(A)$  satisfies the assumptions in Lemma
\ref{lem:supppropertynotdelta}.  In particular, the set of possible $\abs{Z(A)}$ is bounded below. But by 
making $V$ sufficiently small we can  bound $\abs{Z(S)}= \abs{Z(\delta)}  $ from above by an arbitrarily small
positive constant. This proves the claim.

It remains to show that there does exist a stable object of class $\delta$ at some point of $U$.
This is implicit in Theorem \ref{bridge}: otherwise, we could deform stability conditions in $U$
to have $Z_0$ as central charge.\footnote{We can make the logic, implicitly contained in
\cite{Bridgeland:K3}, more explicit as follows. Up to shift, the action of $\Aut_0 \D^b(X)$, and 
the choice of point in $U$, we may assume that $\sigma = (Z, \P) \in U$ is a geometric stability
condition with $Z(v) = -1$ and $\Re Z(\delta) > 0$. Then the slope-stable spherical sheaf of class
$\delta$ (which exists by \cite[Theorem 0.1]{Yoshioka:Irreducibility}) is automatically
$\sigma$-stable by the construction of geometric stability conditions.}
\end{proof}

For general reasons (see \cite[Prop. 9.3]{Bridgeland:K3}), any compact subset of $\Stab(X)$ meets only finitely many walls for our fixed object $E$. Now, however, we need a similar result for certain non-compact subsets of $\Stab(X)$ whose images in $\Hom_\Z(\N(X),\C)$ contain  $Z_0$ in their closure. 

\begin{lemma} \label{lem:finitewalls}
Let $[\phi_1, \phi_2]$ be an interval of length less than $2$, and $\|\cdot \|$ some norm
on $\N(X) \otimes \R$. Then there exists $\epsilon > 0$ with
the following property: if $V_{\epsilon, \phi_1, \phi_2} \subset \P^+_0(X)$ is the subset of central
charges $Z$ satisfying
\[ \| Z-Z_0 \| \le \epsilon \quad \text{and} \quad
Z(\delta) \in \R_{>0} \cdot e^{i\pi\phi} \text{ for some $\phi \in [\phi_1, \phi_2]$}, \]
and if $U_{\epsilon, \phi_1, \phi_2}\subset \Stab^*(X)$ is any connected component of the preimage of $V_{\epsilon,
\phi_1, \phi_2}$, then there are only finitely many $(\pm)$ or integral walls intersecting
$U_{\epsilon, \phi_1, \phi_2}$.

Moreover, if we choose $\epsilon$ small enough, then
$Z(\delta)$ is aligned with $Z(A_+)$ or $Z(A_-)$, respectively,  along any
$(\pm)$ wall contained in $U_{\epsilon, \phi_1, \phi_2}$.
\end{lemma}

\begin{proof}
We choose constants $\epsilon, K$ as in Lemma \ref{lem:supppropertynotdelta}.

Given a stability condition $\sigma$, recall that the \emph{mass} of our object $E$ with respect to
$\sigma$ is defined by
\[
m_E(\sigma) := \sum_i \abs{Z(A_i)}.
\]
This is a continuous function on $\Stab(X)$ (\cite[Proposition
8.1]{Bridgeland:Stab}).
\begin{description}
\item[Claim] The mass $m_E(\sigma)$ is bounded on $U_{\epsilon, \phi_1, \phi_2}$. 
\end{description}

To prove the claim, first note that by compactness, it evidently holds on the intersection $U_N$ of
$U_{\epsilon, \phi_1, \phi_2}$ with the set defined by $\abs{Z(\delta)} \ge \frac 1N $ for
any $N > 0$. We choose $N$ large enough such that $U_N$ contains central charges $Z$ such
that $Z(\delta)$ attains any possible phase $\phi \in [\phi_1, \phi_2]$.
 
Let $\sigma_i = (\P_i, Z_i)\in U_N$
be two stability conditions
that are contained in the same chamber for $E$, and that satisfy
$\abs{Z_2(\delta)} \le \abs{Z_1(\delta)}$. Here, by chamber, we mean a connected component of the complement of all walls in $U_N$ for the object $E$.  Thus the stable factors $A_i$  of $E$ are constant in the
interior of the chamber, and so
\begin{align*}
m_E(\sigma_2) &= \sum_{i=1}^m \abs{Z_2(A_i)} 
\le \sum_i \abs{Z_1(A_i)} + \sum_{i \colon v(A_i) \neq \pm\delta} \abs{(Z_2 - Z_1)(A_i)}  \\
&\le m_E(\sigma_1) + \sum_{i \colon v(A_i) \neq \pm\delta} \|Z_2 - Z_1 \| \cdot \|v(A_i)\|  \\
&\le m_E(\sigma_1) + \|Z_2 - Z_1 \| \cdot \frac 1K m_E(\sigma_1)
\le m_E(\sigma_1) e^{\frac {\|Z_2 - Z_1 \|}K}
\end{align*}
It follows by continuity and induction on the number of chambers traversed that if stability conditions $\sigma_1, \sigma_2\in U_N$ can be
connected by path of length $D$ along which $\abs{Z(\delta)}$ is decreasing, then 
\[
m_E(\sigma_2) \le m_E(\sigma_1) e^{\frac DK}.
\]
Taking the limit as $N\to \infty$, the same result holds for all $\sigma_1, \sigma_2\in U_{\epsilon, \phi_1, \phi_2}$

Let $M$ be the maximum of $m_E$ on the compact set $U_N$. Any point 
$\sigma' \in U_{\epsilon, \phi_1, \phi_2}$ can be reached from a point in $U_N$ along a path of bounded length
$\le 2\epsilon$ along which $\abs{Z(\delta)}$ is decreasing: indeed, the subset of $U_{\epsilon, \phi_1,
\phi_2}$ where $Z(S)$ has constant phase is connected, and (by assumption on $N$) contains a point
$\sigma''$ that is also in $U_N$; then we can just use the straight line segment between the central
charges $Z''$ and $Z'$ of $\sigma''$ and $\sigma'$, respectively. 
It follows that $m_E$ is bounded on
$U_{\epsilon, \phi_1, \phi_2}$ by $M e^{\frac {2\epsilon} K}$, proving the claim.

There are only finitely many $v \in N(X)$ with $(v,v) \ge -2$,
such that $\abs{Z(v)} \le M e^{\frac DK}$ can hold for some $Z \in V_{\epsilon, \phi_1, \phi_2}$.
Therefore, there are only finitely many classes that can
appear as the Mukai vector of a stable factor of $E$ for \emph{any} stability condition in
$U_{\epsilon, \phi_1, \phi_2}$. The loci where pairs of these classes have equal phase defines a
finite set of walls. 

It remains to prove the last statement. Since the number of walls is finite, we can choose
$\epsilon$ small enough such that the image of each wall in $V_{\epsilon, \phi_1, \phi_2}$
contains $Z_0$ in its closure. The wall is locally defined by $Z(S_1)/Z(S_2) \in \R$ for two
stable (semi)rigid facts $S_1, S_2$ of $A_\pm$. This is equivalent to the condition that the
(always one-dimensional) kernel of $Z$ is contained in the span of $v(S_1), v(S_2)$. If $Z_0$
is contained in its closure, it satisfies the same condition. As the kernel of $Z_0$
is spanned by $\delta$, this means that $\delta$ is a linear combination of the $v(S_i)$;
this proves the claim.
\end{proof}

%

%
%

\begin{prop} \label{prop:flow}
Consider a stability condition $\sigma$ with $w_E(\sigma) \notin \Z$. Then the
flow of the vector field $v(\sigma)$ starting at $\sigma$ ends at a stability condition
of integral width $\lfloor w_E(\sigma) \rfloor$.
\end{prop}

\begin{proof}
Assume otherwise. By the results of Section \ref{sec:flow}, this means that the path $\Omega(t) \in \Q^+(X)$ leads to a point
$\Omega_0$ as above. By Lemma \ref{lem:small},
there is a spherical object $S$ with $Z_0(S) = 0$ that is $\sigma(t)$-stable for $t$ sufficiently
close to $t_0$. Replacing $S$ by a suitable shift, we may assume that $(\Theta, v(S)) > 0$ along the path.

Take $0<t_1 < t_0$  as in Lemma \ref{lem:zetacone} and put $\sigma(t_1)=(Z_1,\P_1)$.
Under the assumptions,  $Z_0(S)-Z_1(S) = -Z_1(S)$ is an integral of a positive multiple of $\zeta(t)$; the Lemma
thus implies that $Z_1(S)$ lies in the interior of the cone spanned by $-Z_1(A_+)$ and $Z_1(A_-)$, see 
Figure \ref{fig:hole}. Since $Z(S)\to 0$ along the flow, it follows from the definition of the sign $\epsilon(\sigma)$ that $v(A_-)$ and $v(S)$ lie in the same component $\CC^\pm$, and $v(A_+)$ in the opposite one.

\begin{figure}

\newcommand{\degre}{\ensuremath{^\circ}}
\definecolor{qqwuqq}{rgb}{0,0.39,0}
\definecolor{uququq}{rgb}{0.25,0.25,0.25}
\begin{tikzpicture}[line cap=round,line join=round,>=triangle 45,x=1.0cm,y=1.0cm]
\clip(-3.5,-1.3) rectangle (3.5,2.5);
\draw [shift={(0,0)},color=qqwuqq,fill=qqwuqq,fill opacity=0.1] (0,0) -- (-21.16:1.3) arc
(-21.16:38.66:1.3) -- cycle;
\draw [->] (0,0) -- (-2.48,0.96);
\draw [->] (0,0) -- (2,1.6);
\draw [line width=0.4pt,dash pattern=on 2pt off 2pt,domain=-3.5:3.5] plot(\x,{(-0-1.6*\x)/-2});
\draw [line width=0.4pt,dash pattern=on 2pt off 2pt,domain=-3.5:3.5] plot(\x,{(-0-0.96*\x)/2.48});
\draw [->] (0,0) -- (-2.48,0.96);
\draw [->,line width=0.4pt] (0,0) -- (2.4,-0.06);
\fill [color=uququq] (0,0) circle (1.5pt);
\fill [color=black] (-2.48,0.96) circle (1.5pt);
\draw[color=black] (-2.28,1.24) node {$Z(A_+)$};
\fill [color=black] (2,1.6) circle (1.5pt);
\draw[color=black] (1.70,1.95) node {$Z(A_-)$};
\fill [color=uququq] (-2.48,0.96) circle (1.5pt);
\fill [color=black] (2.4,-0.06) circle (1.5pt);
\draw[color=black] (2.54,0.2) node {$Z(S)$};
\end{tikzpicture}

\caption{The situation at a hole} \label{fig:hole}
\end{figure}
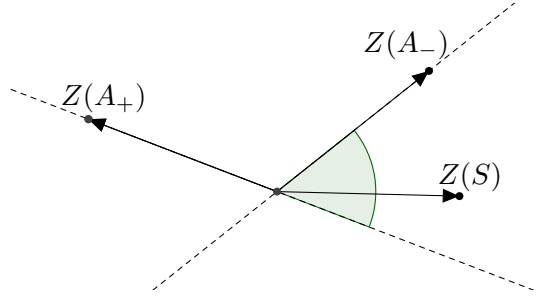

Let the phases of $A_+[-1]$ and $A_-$ in the stability condition $\sigma_1$ be $\phi_1$ and
$\phi_2$, respectively, and let $\epsilon >0$ be sufficiently small such that all conclusions of Lemma
\ref{lem:finitewalls} hold for $V = V_{\epsilon, \phi_1, \phi_2}$. Since $Z(S)\to 0$ under the flow, we can increase $t_1$ if necessary so that for all $t\in [t_1,t_0]$ the central charge of the stability condition $\sigma(t)$ lies in $V$. Thus, choosing the connected component $U=U_{\epsilon,\phi_1,\phi_2}$ appropriately we can assume that $\sigma(t)\in U$ for $t\in [t_1,t_0]$. By Lemma \ref{lem:small} we can also assume that $S$ is stable for all $\sigma\in U$. Finally, replacing $S$ by an even shift, we may  assume that the phase of $S$ in the stability conditon $\sigma_1$ lies in the interval $(\phi_1,\phi_2)$. 

Consider moving $\sigma$ on a path above the open subset $\| Z-Z_0 \| < \epsilon$ starting at $\sigma_1$ by  rotating $Z(S)$ anti-clockwise around the hole. 
The support condition \eqref{eq:supportnotdelta} implies that for $Z, Z' \in
V$ and $(v,v) \ge -2$ with $ v \neq \pm \delta$
\[
\frac{\abs{Z'(v) - Z(v)}}{\abs{Z(v)}} \le \frac{\|Z-Z'\| \|v\|} {K\|v\|} \le 2\epsilon/K.
\]
Thus, after reduing $\epsilon$ further if necessary, the phase of $Z(v)$ is bounded within an arbitrarily small
interval, whereas the phase of $Z(\delta)$ can be increase by an arbitrary amount; in particular,
there is a finite length path such that at the endpoint, the phase of $Z(S)$ aligns with that
of $Z(A_-)$, while it is never aligned with the phase of any of $\pm Z(A_+)$ or $-Z(A_-)$.

Suppose that this path hits an integral wall $\W$.  By the proof of Lemma \ref{wall2}, the stable factors of $A_\pm$ in the stability conditions before we hit the wall are the two stable objects $S_1$, $S_2$ defining the wall, at least one of which is rigid. The last statement of Lemma \ref{lem:finitewalls} shows that $Z(S)$ also aligns with the $Z(S_i)$, and since $S$ is definitely not a stable factor of $A_\pm$, this then contradicts  Lemma \ref{ynew}.

Thus we can prolong our path until  $Z(S)$ aligns with $Z(A_-)$.
There are now
two possibilities: if $S$ and $A_-$ are both stable, then $(v(S), v(A_-)) \ge 0$ implies that $v(S)$ and $v(A_-)$ lie in opposite
components $\CC^\pm$, a contradiction.  Otherwise, if $(v(S), v(A_-)) < 0$, then since
$\Hom_X(A_-, S) = 0$ by stablity we must have $\Hom_X(S, A_-) \neq 0$. Therefore,
we are on a $(-)$-wall. Let $T$ be the Jordan-H\"older factor of $A_-$ that is non-isomorphic
to $S$ (which is unique by Proposition \ref{cases}). After crossing this wall, the last step of the Harder-Narasimhan filtration of $A_-$ induces a surjection
\[A_- \onto T^{\oplus k} \]
for $k = \dim \Hom_X(A_-, T) > 0$, and $T^{\oplus k}$ has become the new $A_-$.  Lemma \ref{cases}
gives $(v(T),v(A_+))> 0$. Since we also have $(v(T),v(S))> 0$ this implies that $v(A_+)$ and $v(S)$ lie in the same component of $\CC^\pm$, which gives another contradiction.   
\end{proof}


\section{Conclusion of the proof} \label{sec:conclusion}
In this section we complete the proofs of our main Theorems. Take assumptions as in the previous sections: thus $X$ is a complex  K3 surface of Picard rank $\rho(X)=1$, and $E\in \D(X)$ is a (semi)rigid object satisfying $\Hom_X(E,E)=\C$.

\subsection*{Retraction onto width 0}
The aim of this section is to combine the flows of Proposition \ref{prop:flow} defined on
$w_E^{-1}(n, n+1)$ for all $n$ to retract $\Stab^*_\red(X)$ onto a subset of the geometric chamber.

We write $W_{\le n}$, $W_{=n}$, $W_{<n}$ and $W_{(n, n+1)}$ for the subsets of $\Stab^*_\red(X)$ defined
by $w_E \le n$, etc.
(Note that $W_\Z$ is a strict subset of  $\cup_{n \in \Z} W_{=n}$, as we make no assumptions
about quasistability of $E$ in $W_{=0}$.)

\begin{lemma}\label{lem:n}
The inclusion $W_{\le n} \subset W_{< n+1}$ is a deformation retract.
\end{lemma}
\begin{proof}
There is a continuous map $W_{(n, n+1)} \to W_{=n}$ sending a stability condition
$\sigma$ with $w_E(\sigma) \in (n, n+1)$ to the endpoint of the flow along the vector field
$v$ starting at $\sigma$.  Lemma \ref{lem:flowatwall} shows that the identity on $W_{\le n}$ extends
this map  to give a continuous retract $r_n \colon W_{< n+1} \to W_{\le n}$. 

A homotopy between the identity on $W_{< n+1}$ and $r_n$ is given by the 
normalized flow map
\[ \mathrm{Flow}' \colon W_{< n+1} \times [0,1] \to W_{\le n}, \] 
for which $\mathrm{Flow}'(\sigma, t)$ is the point on the
flow line of $\sigma$ with width $w_E(\sigma) - t$ if $0 \le t < w_E(\sigma) - n$, and given by 
$r_n(\sigma) \in W_{\le n}$ if $t \ge w_E(\sigma) - n$.

We now make this construction rigorous. Consider the union $U \subset W_{(n, n+1)} \times \R_{\ge 0}$
of all maximal (closed) intervals of definition of the flow of $v(\sigma)$, and let
\[ \mathrm{Flow} \colon U \to W_{[n, n+1)}
\]
be the the induced continuous map: $\mathrm{Flow}(\sigma, t)$ is the position of the flow line starting
at $\sigma$ after time $t$. It follows from construction and
Lemma \ref{lem:decreasingwidth} that the map
\[
\Gamma \colon U \to W_{(n, n+1)} \times [0, 1), \quad
(\sigma, t) \mapsto \bigl(\sigma, w_E(\sigma) - w_E(\mathrm{Flow}(\sigma, t))\bigr) 
\]
is a homeomorphism onto its image
\[
V = \left\{(\sigma, s) \colon 0 \le s \le w_E(\sigma) - n\right\} \subset W_{(n, n+1)} \times [0, 1).
\]
Thus, on $W_{(n, n+1)} \times [0,1]$, we can define $\mathrm{Flow}'$
as follows: first set
$\mathrm{Flow}'|_V = \mathrm{Flow} \circ \Gamma^{-1}$, and then extend it
continuously by $\mathrm{Flow'}(\sigma, s) = r_n(s)$ for $(\sigma, s) \notin V$.
This is a homotopy between the inclusion $W_{(n, n+1)} \into \Stab^*_\red(X)$ and $r_n$.

Another application of Lemma \ref{lem:flowatwall} shows that this homotopy extends continuously 
to $W_{< n+1}$, as desired.
\end{proof}

\begin{lemma}\label{lem:n1}
For each $n \ge 0$, there is a subset $U \subset W_{<n+1}$, containing $W_{\le n}$, such that
$U$ is a deformation retract of $W_{\le n+1}$.
\end{lemma}
\begin{proof}
By Lemma \ref{wall2}, the subset $W_{=n+1}\subset \Stab^*_\red(X)$ is a real codimension one submanifold.  By Lemma \ref{lem:integerwalls}, this submanifold borders  $W_{< n+1}$ on only one side. Thus $W_{\le n+1}$ is a manifold with boundary, and the boundary is $W_{=n+1}$. Therefore,
we can find a tubular neighbourhood of $W_{=n+1}\subset W_{< n+1}$ homeomorphic to $W_{= n+1} \times [0, \epsilon)$  and let $U$ be its complement. 
\end{proof}

Combining Lemmas \ref{lem:n} and \ref{lem:n1}, we obtain a retraction
\[ R_n \colon W_{\le n+1} \to W_{\le n}, \]
and a homotopy 
\[ H_n \colon [0, 1] \times W_{\le n+1} \to W_{\le n+1} \]
between the identity and $R_n$.

\begin{lemma} \label{lem:W0retract}
The subset $W_{=0} \subset \Stab^*_\red(X)$ is a deformation retract.
\end{lemma}
\begin{proof}
We define a left inverse $R_\infty$ to the inclusion by the infinite composition
\[
R_\infty = R_0 \circ R_1 \circ R_2 \circ \dots \colon \Stab^*_\red(X) \to W_{=0}.
\]
Of course, on each $W_{<n+1}$, this map is a finite composition and continuous. It follows
that it is well-defined and continuous on $\Stab^*_\red(X)$.

Similarly, to define the homotopy choose any infinite decreasing sequence $t_0 = 1 > t_1 > t_2 >
\dots > 0$; we can define a homotopy $H_\infty$
between the identity and $R_\infty$ as the infinite composition of the homotopies that apply
$H_n$ in the interval $[t_{n+1}, t_n]$. The same argument as before shows that
$H_\infty$ is well-defined and continuous after restriction to $[0,1] \times W_{\le n+1}$,
and therefore well-defined and continuous on all of $[0,1] \times \Stab^*_\red(X)$.
\end{proof}

\subsection*{Geometric stability conditions}
We now fix our object $E$ to be some skyscraper sheaf $E=\O_x$. To prove Theorem \ref{main} it
remains to prove that $W_{=0}$ is connected and contractible.  In fact, given our assumption that the Picard rank
$\rho(X)=1$, the interior of the subset $W_{=0}$ coincides with the subset of geometric stability conditions,
as we now demonstrate.

\begin{lemma}\label{lem:semirigidsheaf}
Assume that $F \in \Coh (X)$ is a semirigid sheaf. Then either
$F \cong \O_x$ for some $x \in X$, or the support of $F$ is all of $X$.
\end{lemma}
\begin{proof}
Skyscrapers are the only semirigid zero-dimensional sheaves, so if neither of the given alternatives hold, then $F$ has one-dimensional support, and so
$v(F)=(0,D,s)$ with $D$ an effective curve class. But this is impossible,  since  $\operatorname{Pic}(X)= \Z$
implies that
\[ (v(F), v(F)) = D^2  > 0\]
contradicting the assumption that $F$ is semirigid.
\end{proof}
\begin{remark} \label{rem:sphericaltorsionfree}
Combined with Lemma \ref{stan}, this immediately shows that any rigid sheaf is
automatically torsion-free.
\end{remark}

The following is due to Hartmann.

\begin{lemma} 
Given $\phi \in \Aut^+ H^*(X, \Z)$, there exists an autoequivalence
$\Phi \in \Aut D(X)$ whose induced map in cohomology is given by $\phi$, and
such that $\Phi$ preserves the connected component $\Stab^\dag(X)$.
\end{lemma}

\begin{proof}
 See \cite[Proposition 7.9]{Hartmann:cusps}.
\end{proof}

We can now prove the claimed characterisation of geometric stability conditions.

\begin{lemma}
\label{ble}
Suppose $\sigma\in \Stab^*(X)$ is such that some $\O_x$ is stable. Then all skyscraper sheaves are stable of the same phase.
\end{lemma}

\begin{proof}
We first prove the result under the additional assumption that $\sigma$ is not on a wall with respect to the class $v(\O_x)$.

By definition of $\Stab^*(X)$, there is an autoequivalence $\Phi\in \Aut \D(X)$ such that
$\Phi(\sigma) \in \Stab^\dag(X)$. By the previous Lemma, we may assume that
$\Phi$ acts as the identity on cohomology; in particular, it leaves
the class $v(\O_x)$ invariant. By the argument used in the proof of \cite[Prop. 13.2]{Bridgeland:Stab}, we can compose with squares of spherical twists,
and further assume that $\Phi(\sigma)$ is in the geometric chamber.
The stable objects in $\sigma$ of class $v(\O_x)$  and the correct phase
are exactly the objects $E_y=\Phi^{-1}(\O_y)$ for $y\in X$. 

The objects $E_y$
are pairwise orthogonal.
By assumption one of these objects is $\O_x$, hence all the others have
support disjoint from $x$. Using Lemma \ref{stan} repeatedly, it follows that each cohomology sheaf of $E_y$ is
rigid or semirigid, and at most one is semirigid. By Lemma \ref{lem:semirigidsheaf},
this is only possible if $E_y$ is the shift of a skyscraper sheaf: $E_y \cong \O_z[n]$ for
some $z \in X$ and $n \in \Z$. Using the representability of $\Phi$ as a Fourier-Mukai transform,
standard arguments (see \cite[Corollary 5.23 and Corollary 6.14]{Huybrechts:FM})
show that $n$ is independent of $y$, and that $\Phi$ is the composition of an
automorphism of $X$ with a shift. In particular, all $\O_y$ are stable of the same phase.

Finally, if $\sigma$ is on a wall for the class $v(\O_x)$, then  by applying the previous 
argument  to a small perturbation of $\sigma$ we can conclude that $\sigma$ is in the boundary of the geometric chamber. However, it follows from
\cite[Theorem 12.1]{Bridgeland:Stab} that in the case of Picard rank one, every wall of the geometric
chamber destabilizes every skyscraper sheaf $\O_x$.
\end{proof}


\subsection*{Final steps}
We can now complete the proofs of our main Theorems.

\begin{lemma} \label{lem:W0contracts}
Let $E=\O_x$ be a sykscraper sheaf and consider the width function $w=w_E$. 
Then $W_{=0}$ is contractible.
\end{lemma}

\begin{proof}
By the previous Lemma, the set $W_{=0}$ coincides with the closure of the geometric chamber.
Recall from Theorem \ref{geom} that its interior is homeomorphic 
to $\C \times \uu$. It is immediate to see that $\uu$ is contractible.
With arguments as in Lemma \ref{lem:n1} one also shows 
that there is an open subset of the geometric chamber that is a deformation retract
of its closure $W_{=0}$.
\end{proof}

\begin{proof}[Proof of Theorem \ref{main}]
The result is immediate from Lemmas \ref{reducedretract}, \ref{lem:W0retract} and
\ref{lem:W0contracts}.
\end{proof}

\begin{proof}[Proof of Theorem \ref{corcor}]
As discussed in the introduction, since $\Stab^*(X)$ is simply-connected, there 
is an isomorphism
\[ \pi_1(\P_0^+(X)) \xrightarrow{\sim} \Aut^0 D(X). \]
Recall that $\pi_1(\P_0^+(X))$ is the product of
$\pi_1 \left(\GL^+_2(\R)\right)  \cong \Z$ with the fundamental
group of $\h^0 \subset \h$, which in turn is a free group generated by loops
around the holes $\delta^\perp$ (see equation \eqref{eq:pi1P0} and the surrounding discussion).

By Proposition \ref{prop:loopandtwist}, these loops act by squares of spherical twists. Finally,
from the definition of the action of the universal cover of ${\GL^+_2}(\R)$, it is obvious that
the generator of $\pi_1(\GL^+_2(\R))$ acts by an even shift.
\end{proof}

To conclude, we point out two consequences of our results for spherical objects;
they may be of independent interest, e.g. in relation to \cite{Huybrechts:spherical}. 
\begin{cor}
Let $S \in D(X)$ be a spherical object. Then there exists a stability condition
$\sigma \in \Stab^\dag(X)$ such that $S$ is $\sigma$-stable.
\end{cor}
\begin{proof}
This is immediate from Lemma \ref{lem:W0retract}, applied to $E = S$.
\end{proof}

\begin{remark}
Based on the above Corollary, one can also prove the following statement: \emph{$\Aut^0 D(X)$ acts
transitively on the set of spherical objects $S$ with fixed Mukai vector $v(S) = \delta$.}

(As a consequence, $\Aut D(X)$ acts transitively on the set of spherical objects in $D(X)$ if and only if
if $\Aut^+ H^*(X, \Z)$ acts transitively on the set $\Delta(X)$ of spherical classes. Similar
results have been proved for spherical vector bundles and mutations in 
\cite{Kuleshov:spherical}.)

By the above Corollary, our claim follows if for any wall $\W$ where stability for the class
$\delta$ changes, we can prove that the two stable objects $T^+, T^-$ of class $v(T^{\pm}) = \delta$
on either side of the wall are related by $T^- = \Phi(T^+)$, for some autoequivalence in $\Phi \in \Aut^0 D(X)$. 
This can, for example, be shown by applying \cite[Proposition
6.8]{BM:walls}. (In the language and notation of [ibid.], we have $v$ equal to our $\delta$,
and $v_0$ is equal to the Mukai vector of one of the two spherical objects
$S, T$ that are \emph{stable} on the wall; these are exactly the two simple objects in 
the category $\A$ considered in Section \ref{sec:walls}. Therefore, up to switching $S$ and $T$
we can set $E_0$ equal to $S$. Then [ibid.] shows that both $T^+$ and $T^-$ are obtained
as spherical twists of $E_0=S$, which implies our claim.)
\end{remark}


\section{Relation with mirror symmetry} \label{sect:MS}

Return to the case of a general algebraic K3 surface $X$. We will describe a precise relation
between the group of autoequivalences and the monodromy group of the mirror family implied by 
Conjecture \ref{kyoto}.

\subsection*{Stringy K{\"a}hler moduli space}

We start by reviewing the construction of an interesting subgroup of $\Aut \D(X)$, 
 which we learnt about from Daniel Huybrechts.
Let us write \[\Aut_{\CY} ^+H^*(X)\subset \Aut ^+ H^*(X)\] for the subgroup of Hodge isometries
$\phi$ whose complexification acts trivially on the line $H^{2,0}(\C)\subset H^*(X,\C)$. This is
equivalent to the statement that $\phi$ acts trivially on the transcendental lattice
$T(X):=\N(X)^{\perp}\subset H^*(X,\Z)$: for any integral class $\alpha \in T(X)$, the difference
$\phi(\alpha) - \alpha$ is integral, and in the orthogonal complement of both $H^{2,0}(\C)$ and
$\N(X)$, and thus equals zero.  In particular,
\[
\Aut^+_{\CY}H^*(X) \subset \Aut \N(X)
\]
is the subgroup of index two preserving orientations of positive definite two-planes.

\begin{defn} \label{def:CYequivalence}
We call an autoequivalence $\Phi\in \Aut \D(X)$ \emph{Calabi-Yau} if the induced Hodge isometry 
$\varpi(\Phi)$ lies in the subgroup $\Aut_{\CY} H^*(X)$.
\end{defn}

Write $\Aut_{\CY} \D(X)\subset \Aut \D(X)$ for the group of Calabi-Yau autoequivalences. In the Appendix we explain how the Calabi-Yau condition can be interpreted as meaning that $\Phi$ is an autoequivalence of the category $\D(X)$ as a Calabi-Yau category.

Now consider  the quotient stack
\[\L_{\rm Kah}(X)=\Stab_{\rm red}^*(X)/\Aut_{\rm CY} D(X).\]
The action of $\C$ on $\Stab^*(X)$ induces an action of $\C^*$ on $\L_{\rm Kah}(X)$ and we also consider the quotient
\[ \M_{\rm Kah}(X)=\L_{\rm Kah}(X)/\C^*\isom (\Stab_{\rm red}^*(X)/ \C)/(\Aut_{\rm CY} D(X)/[2]).\]
 which we view as  a mathematical version of the stringy K{\"a}hler moduli space of the K3 surface $X$.

\begin{remark} If 
 Conjecture \ref{kyoto} holds  then $\Stab^*_{\rm red}(X)$ is the universal cover of $\L_{\rm red}(X)$, and there are isomorphisms
\begin{equation*}
\label{do} \pi_1^{\rm orb} (\L_{\rm Kah}(X)) \isom  \Aut_{\CY} D(X), \quad  \pi_1^{\rm orb} (\M_{\rm Kah}(X)) \isom  \Aut_{\CY} D(X) /[2].\end{equation*}
In particular, by Theorem \ref{main}, this is the case whenever  $\rho(X)=1$.\end{remark}

We have the following more concrete descriptions of these stacks

\begin{prop}
There are isomorphisms
\[\L_{\rm Kah}(X)\isom \Q_0^+(X)/\Aut ^+_\CY H^*(X),\quad \M_{\rm Kah}(X)\isom \Omega_0^+(X)/\Aut ^+_\CY H^*(X).\]
Moreover, these stacks are Deligne-Mumford and have quasi-projective coarse moduli spaces.
\end{prop}

\begin{proof}
There is a short exact sequence
\[ 1\lra \Aut^0 \D(X) \lra \Aut_{\CY} D(X) \lra \Aut^+_{\CY} H^*(X)\lra 1.\] 
Together with Theorem \ref{bridge} this leads to the given isomorphisms.
The stabilizer of $\Omega \in \Q_0^+(X)$ acts faithfully on $\Omega^\perp \cap \N(X)$; since
$\Re \Omega, \Im \Omega$ span a positive definite 2-plane in $\N(X) \otimes \R$, this lattice
is negative definite, and thus has finite automorphism groups. Therefore,
the stabilizer of $\Omega \in \Q_0^+(X)$ is finite.
Also recall from above that $\Aut ^+_\CY H^*(X)$ has finite index two in $\Aut \N(X)$. 
Thus, the Baily-Borel theorem applies, and the above stacks have quasi-projective coarse moduli spaces.
\end{proof}


\subsection*{Mirror families}
We now relate the space $\M_{\rm Kah}(X)$ to moduli spaces of lattice-polarised K3 surfaces. For this
we need to make

\begin{assumption} Suppose that the transcendental lattice $T(X)$ contains a sublattice isomorphic to the hyperbolic plane.
\end{assumption}

This condition is automatic if $\rho(X)=1$  cf. \cite[Section 7]{Dolgachev:MSK3}. Note that the lattice $\N(X)$ contains a canonical sublattice $H=H^0(X,\Z)\oplus H^4(X,\Z)$ isomorphic to the hyperbolic plane.  
Given the assumption, we can choose another such sublattice $H'\subset T(X)$. Then there are orthogonal direct sums
\[\N(X)=H\oplus M, \quad T(X)=H'\oplus M^\check,\]
where   $M=\NS(X)$ and $M^\check$ is some lattice of signature $(1,18-\rho)$.

Now recall the notion of an ample $M^\check$-polarized K3 surface from \cite[Section 1]{Dolgachev:MSK3};
this includes the data of a K3 surface $\hat X$ together with a primitive isometric
embedding $\rho \colon M^\check \to \NS(\hat X)$ whose image contains an ample class. (The notion
depends on additional choices of data; different choices are equivalent up to pre-composing the
embedding $\rho$ with an isometry of $M^\check$.)

\begin{remark}
There is no separated moduli stack of ample $M^\check$-polarized K3 surfaces for the following well-known reason:
Consider a smooth family $Y \to B$ of K3 surfaces over a one-dimensional base $B$, and assume that
a special fiber $b \in B$ contains a $(-2)$-curve $C$ that does not deform as an algebraic class. Then flopping at $C$
produces a non-isomorphic family $\hat Y \to B$ that is isomorphic to $Y$ after restricting to the
complement of $b$. Note that the central fibers are isomorphic as K3 surfaces, but 
not isomorphic as ample $M^{\check}$-polarized K3 surfaces.
\end{remark}

\begin{lemma}
The orbifold $\M_{\rm Kah}(X)$ admits a family of $M^\check$-polarized K3 surfaces, and its coarse
moduli space is the coarse moduli space $M^\check$-polarized K3 surfaces.
\end{lemma}

\begin{proof}
 Consider the orthogonal complement $L=H'^\perp\subset H^*(X,\Z)$. Note that $L$ is isomorphic to the K3 lattice $H^2(K3,\Z)$. We have an orthogonal decomposition (not necessarily a direct sum)
\[L=H \oplus M \perp M^\check.\]
In particular, inside $L$ we have  $(M^\check)^\perp= \N(X)$. Note also that $\Aut_{CY} H^*(X,\Z)$ can be identified with the group of automorphisms of $L$ which fix every element of $M^\check$. 

This is exactly the situation considered by Dolgachev in \cite[Section 6]{Dolgachev:MSK3}: in terms
of his notation $\N(X)$ becomes identified with the lattice $N$, the space $\PP Q^0(X)$ becomes
$D_{M^\check}^o$ and the group $\Aut_\CY H^*(X,\Z)$ becomes $\Gamma(M^\check)$. In particular, the 
statement regarding coarse moduli is proved there.
In order to construct a family, let us choose in addition a class 
$l \in M^\check$ with $l^2 > 0$ such that $l$ is not orthogonal to any spherical class
$\delta \in L \setminus \N(X)$; requiring $l$ to be ample avoids the non-Hausdorff issue explained
above.

Now we use the strong global Torelli theorem: given $\Omega \in \Q_0^+(X)$, there exists a unique
K3 surfaces $\hat X$ with a marking $L \xrightarrow{\sim} H^2(\hat X, \Z)$
such that $H^{2,0}(X, \C)$ is spanned by $\Omega$, and such that $l$
is an ample class. These fit together into a family over the period domain, on which 
$\Aut_\CY H^*(X, \Z)$ acts. Taking the quotient by this action produces a family 
over $\M(X^\check)$, and remembers the marking by the sublattice $M^\check$ as claimed.
\end{proof}

Following Dolgachev, we consider this family of $M^\check$-polarised K3 surfaces as a mirror family to the family of (ample) $M$-polarized K3 surfaces, of which our surface $X$ is a member.
Thus in the case $\rho=1$ we can conclude that the group $\Aut_\CY D(X) /[2]$ is isomorphic to the fundamental group of the base of the mirror family.
Alternatively, note that  the full group $\Aut_\CY D(X)$ is isomorphic to the fundamental group of this augmented mirror moduli space
$\L_{\rm Kah}(X)$ parameterizing pairs consisting of an ample $M^\check$-polarized K3 surfaces together with a choice of nonzero holomorphic 2-form.

 \begin{remark}The lattice $M^\check$ and its embedding in the K3 lattice $L$  depends on our choice $H'\subset T(X)$. Different choices lead to different equivalence classes of embeddings and hence different families of $M^\check$-polarised K3 surfaces. The bases of these  families are all identified with the space $\M_{\rm Kah}(X)$, but as families of $M^\check$-polarised K3 surfaces they are different. All should be considered as mirror families of $X$. It is easy to check that  the families of derived categories given by these different mirror families are all the same.    \end{remark}


\begin{remark} \label{rmk:AutFuk}
Finally, we want make to make explicit a 
relation of Conjecture \ref{kyoto} to Homological Mirror Symmetry.
Given an ample $M^\check$-polarized K3 surface $\hat X$, and a K\"ahler form $\omega$, assume that
Homological Mirror Symmetry holds for $X$ and $(\hat X, \omega)$: we assume that
\[ D(X) \cong \mathrm{Fuk}(\hat X, \omega) \]
where $\mathrm{Fuk}(\hat X, \omega)$ is a suitably defined Fukaya category.

Let $\pi_1(\M_{\rm Kah}(X))$ be the orbifold fundamental group of the base of our mirror family of
$M^\check$-polarized K3 surfaces. If Conjecture \ref{kyoto} holds in addition to homological mirror
symmetry, we obtain an isomorphism
\[
\pi_1(\M_{\rm Kah}(X)) \cong \Aut \mathrm{Fuk}(\hat X, \omega).
\]
Presumably, any proof of Homological Mirror Symmetry would allow us to identify this isomorphism
with a map induced by the monodromy in the mirror family.

\end{remark}

\appendix
\section{Calabi-Yau autoequivalences}
In this Appendix we explain that an autoequivalence $\Phi\in \Aut \D(X)$ is Calabi-Yau in the sense used above precisely if
$\Phi$ respects the Serre duality pairings
\begin{equation} \label{eq:Serreduality}
\Hom^i(E, F) \times \Hom^{-i}(F, E[2]) \to \C
\end{equation}
induced by a choice of holomorphic volume form $\omega \in H^{2,0}(X, \C)$.

A proof would be more natural in the setting of dg-categories; thus we restrict ourselves to a
sketch of the arguments.
First we recall some basic definitions and results on Serre functors from \cite[Section 1]{Bondal-Orlov}:
\begin{itemize}
\item A \emph{graded autoequivalence} is an autoequivalence $\Phi$ together with a natural transformation
$\Phi \circ [1] \Rightarrow [1] \circ \Phi$.
Any exact autoequivalence has the structure of a graded autoequivalence.
\smallskip

\item A \emph{Serre functor} is a graded autoequivalence $S$ together with functorial isomorphisms
\[ 
\Hom(A, B) \cong \Hom(B, S(A))^*
\]
satisfying an extra condition.\smallskip

\item A Serre functor is unique up to a canonical graded natural transformation.
\smallskip

\item Given an equivalence
$\Phi$, there is a canonical natural isomorphism 
\[ \Phi \circ S \Rightarrow S \circ \Phi. \]
\end{itemize}

Let us define a Calabi-Yau-2 category to be a triangulated category together with natural isomorphism
$\tau \colon [2] \Rightarrow S$, where $S$ is a Serre functor. By the canonical uniqueness of Serre functors, this is equivalent
to specifying functorial Serre duality pairings as in \eqref{eq:Serreduality}. 

A graded autoequivalence acts on the set of natural transformations $[2] \Rightarrow S$ via
conjugation, and the given natural transformations $\Phi \circ [2] \Rightarrow [2] \circ \Phi$
and $\Phi \circ S \Rightarrow S \circ \Phi$. From the construction of the latter transformation, it
follows that $\Phi$ leaves $\tau$ invariant if and only if it respects the Serre duality pairings
\eqref{eq:Serreduality}.

On the other hand, the induced actions of $\Phi$ on the cohomology $H^*(X, \C)$ of $X$ and
the Hochschild homology of $\D(X)$ are compatible with the HKR isomorphism, see
\cite[Theorem 1.2]{Emanuele-Paolo:infinitesimal}. Therefore,
$\Phi$ is Calabi-Yau in the sense of Definition \ref{def:CYequivalence} if and only if it
leaves the second Hochschild homology group
\begin{equation} \label{eq:hochschild}
\mathrm{HH}_2 (X) = \Hom_{\D(X \times X)}\left(\O_{\Delta}[2], \omega_{\Delta}[2]\right)
\end{equation}
invariant.

Passing to Fourier-Mukai transforms induces a natural map
\[
\Hom_{\D(X \times X)}(\O_\Delta[2], \omega_\Delta[2]) \to \Hom([2], S),
\]
compatible with the action of $\Phi$.
While it is \emph{not} true (without passing to dg-categories) for any pair of kernels that the
above map is an isomorphism, it does hold in our situation: both sides are (non-canonically)
isomorphic to scalars $\C$, and the map is non-trivial.
Thus, $\Phi$ is Calabi-Yau in the sense of Definition \ref{def:CYequivalence} if and only if it 
respects the natural transformation $[2] \Rightarrow S$, or, equivalently, respects the Serre
duality pairings \eqref{eq:Serreduality}.

\bibliography{all}  
\bibliographystyle{halpha}

\end{document}